\documentclass[11pt,reqno]{amsart}
\usepackage{amsfonts, amsmath, amssymb, amscd, amsthm, bm}
\usepackage{amsrefs}
\def\qedbox{\hbox{$\rlap{$\sqcap$}\sqcup$}}
\def\qed{\nobreak\hfill\penalty250 \hbox{}\nobreak\hfill\qedbox}
\usepackage{enumerate}
\allowdisplaybreaks[4]
\usepackage{cases}
\usepackage{url}
\usepackage[linktocpage=true,colorlinks,citecolor=magenta,linkcolor=blue,urlcolor=magenta]{hyperref}

\usepackage{multicol}
\usepackage{comment}
\usepackage[title]{appendix}

\newtheorem{thm}{Theorem}[section]

\newtheorem{lem}[thm]{Lemma}
\newtheorem{prop}[thm]{Proposition}
\theoremstyle{definition}

\theoremstyle{theorem}
\newtheorem{rem}[thm]{Remark}
\newtheorem{ex}[thm]{Example}

\theoremstyle{claim}

\numberwithin{equation}{section}
\def\p{\partial}
\def\R{\mathbb{R}^{n+p}(c)}
\def\H{\mathbb{H}^{n+p}(-1)}
\def\S{\mathbb{S}^{n+p}(1)}
\def\et{\tilde{e}}
\def\th{\tilde{h}}
\def\tH{\tilde{H}}
\def\tg{\tilde{g}}
\def\tom{\tilde{\omega}}
\def\tt{\tilde{\theta}}
\def\tn{\widetilde{\nabla}}
\def\tR{\tilde{R}}

\def\la{\lambda}

\def\l{\langle}
\def\r{\rangle}

\def\Ric{\mbox{Ric}}
\DeclareMathOperator\tr{tr}
\DeclareMathOperator\arsinh{arsinh}
\begin{document}

\title{Sharp Reilly-type inequalities for submanifolds in space forms}

\author[H. Chen ]{Hang Chen}
\address{Department of Applied Mathematics, Northwestern Polytechnical University, Xi' an 710072,  P. R. China}
\email{\href{mailto:chenhang86@nwpu.edu.cn}{chenhang86@nwpu.edu.cn}}
%\thanks{The first author was partially supported by NSFC Grant No. 11601426 and the Fundamental Research Funds for the Central Universities Grant No. 3102014JCQ01071.}

\author[X. Wang]{Xianfeng Wang}
\address{School of Mathematical Sciences and LPMC,
Nankai University,
Tianjin 300071,  P. R. China}
\email{\href{mailto:wangxianfeng@nankai.edu.cn}{wangxianfeng@nankai.edu.cn}}
%\thanks{The second author was partially supported by NSFC  Grant No. 11571185.}

\begin{abstract}
Let $M$ be an $n(>2)$-dimensional closed orientable submanifold in an $(n+p)$-dimensional space form $\R$. We obtain an optimal upper bound for the second eigenvalue of a class of elliptic operators on $M$ defined by $L_{T}f=-\mbox{div}(T\nabla f)$, where $T$ is a general symmetric, positive definite and divergence-free $(1,1)$-tensor on $M$. The upper bound is given in terms of an integration involving $\tr T$ and $|H_T|^2$, where $\tr T$ is the trace of the tensor $T$ and $H_T=\sum_{i=1}^nA(Te_i,e_i)$ is a normal vector field associated with $T$ and the second fundamental form $A$ of $M$. Furthermore, we give the sufficient and necessary conditions when the upper bound is attained. Our main theorem can be viewed as an extension of the famous ``Reilly inequality''. The operator $L_{T}$ can be regarded as a natural generalization of the well-known operator $L_r$ which is the linearized operator of the first variation of the $(r+1)$-th mean curvature for hypersurfaces in a space form. As applications of our main theorem, we generalize the results of  Grosjean (\cite{Gro00}) and Li-Wang (\cite{LW12})  for  hypersurfaces to higher codimension case.
\end{abstract}



\maketitle

\section{Introduction}
Let $M$ be an $n(\geq2)$-dimensional closed orientable submanifold immersed in $\R$, where $\R$ is the $(n+p)$-dimensional simply connected space form of constant curvature $c$ and  represents the Euclidean space $\mathbb{R}^{n+p}$, the hyperbolic space $\H$ and the unit sphere $\S$ for $c=0, -1$ and $1$ respectively.

In \cite{ESI00}, A. El Soufi and S. Ilias obtained a sharp upper bound for the second eigenvalue of the Schr{\"o}dinger operator $L_s=-\Delta+q$, where $\Delta$ is the Laplacian on $M$ and $q\in C^{\infty}(M)$. They proved that
\begin{equation}\begin{aligned}\label{SI}
\lambda_2^{L_s}\leq\frac{1}{V(M)}\int_M[n(H^2+c)+q]\,dv_M,
\end{aligned}\end{equation}
where $H$ is the length of the mean curvature vector. Moreover, if $n\geq3$, then equality holds in \eqref{SI} if and only if $q$ is constant and $M$ is minimally immersed in a geodesic sphere of radius $r_c$ of $\R$  with $r_0=\big(n\big/\lambda_2^{\Delta}\big)^{1/2}, r_1=\arcsin r_0$ and $r_{-1}=\arsinh r_0$. When $q=0$, (\ref{SI}) becomes
\begin{equation}\begin{aligned}\label{Rei}
\lambda_2^{\Delta}\leq\frac{1}{V(M)}\int_M[n(H^2+c)]\,dv_M,
\end{aligned}\end{equation}
which is called ``Reilly inequality" as it was first proved by R. Reilly for submanifolds of $\mathbb{R}^{n+p}$ in  \cite{Rei77} and was proved by A. El Soufi and S. Ilias
in \cite{ESI92} for submanifolds of $\H$.  If $M$ is a minimal hypersurface or a hypersurface with constant mean curvature in $\mathbb{R}^{n+1}(c), q=-S-nc$,  where $S$ denotes the squared norm of the second fundamental form, then $L_s$ is the Jacobi operator of $M$ and its spectral behavior is related to the stability of $M$.

Later, J.-F. Grosjean \cite{Gro00} generalized the ``Reilly inequality" (cf.  \cite{Rei77} and \cite{ESI00}) to the $L_r$ operator on hypersurfaces of space forms.
Here  $T_r$  is the $r$-th Newton transformation arising from the shape operator $A$, and  the operator $L_r$  is the linearized operator of the first variation of the $(r+1)$-th mean curvature arising from variations of an immersed hypersurface $M\subset\mathbb{R}^{n+1}(c)$ (see Section 2.3 for details). Under some assumptions, $L_r$ is elliptic and its spectrum is closely related to the stability of hypersurfaces with constant $r$-th mean curvature in space forms. For more details of the properties of the $L_r$ operator, we refer the readers to \cites{AdCR93,CL07,CY77,Rei73,Ros93} and the references therein. In \cite{Gro00}, Grosjean proved that

\noindent\textbf{Theorem A.}~(Theorem 1.1 in \cite{Gro00})
	Let $(M, g)$ be a closed orientable Riemannian manifold of dimension $n\geq2$ and $x$ be a convex isometric immersion of $(M, g)$ in $\mathbb{R}^{n+1}(c)$. If $H_{r+1}>0$ for $r\in\{0, \ldots, n-2\}$, then we have
\begin{equation}
	\begin{aligned}\label{eqG}
	\la_2^{L_r}\leq\frac{(n-r){n\choose r}}{V(M)}\int_M\frac{H_{r+1}^2+cH_r^2}{H_r}\,dv_M.
	\end{aligned}
\end{equation}
	Moreover, equality holds if $x(M)$ is an umbilical sphere. And for $r<n-2$, if equality holds, then $x(M)$ is an umbilical sphere.

When $r=1$,  there exists a natural operator $J_s$ associated with $L_1$, more precisely, $J_s=L_1-S_1S_2+3S_3-(n-1)S_1$ and $L_1(f)=-\sum_{i,j}(nH\delta_{ij}-h_{ij})f_{ij}, ~\forall f\in C^{\infty}(M)$, where $f_{ij}$ is the second covariant derivative of $f$, $h_{ij}$ is the component of the second fundamental form of the hypersurface, $S_r$ is the $r$-th ($r=1,2,3$) mean curvature of the hypersurface, see Section 2.3 for the definition. $J_s$ is the Jacobi operator of hypersurfaces with constant scalar curvature in a unit sphere $\mathbb{S}^{n+1}(1)$. In \cite{LW12}, H. Li and the second author obtained a sharp estimate for the second eigenvalue of $J_s$ without the convexity assumption, they proved that

\noindent\textbf{Theorem B.}~(Theorem 1.4 in  \cite{LW12})
	Let $M$ be an $n$-dimensional closed orientable hypersurface  with
	constant scalar curvature $n(n-1)r=n(n-1)(1+H_2),~H_2> 0$, in ~$\mathbb{S}^{n+1}(1), n\geq5.$ Then the Jacobi operator $J_s$ is elliptic and its second eigenvalue $\la_2^{J_s}$ satisfies
	\begin{equation*}
	\la_2^{J_s}\leq0,
	\end{equation*}
	where the equality holds if and only if $M$ is totally umbilical and non-totally geodesic.

In fact, from the proof of Theorem B, one can see that the authors did not use the condition of ``$H_2$ being constant'' and one can get the same estimate as in \eqref{eqG} for $L_1$ without the convexity assumption by adapting the arguments in \cite{LW12}.

We notice that the inequalities  \eqref{SI} and \eqref{Rei} were proved for submanifolds in space forms, while the results in Theorem A and Theorem B are in the setting of hypersurfaces in space forms, a natural question is the following:

\noindent\textbf{Problem.} In higher codimension case, can we get sharp inequalities analogous to \eqref{SI} and \eqref{Rei} for some general elliptic operators? Can we generalize Theorem A and Theorem B to higher codimension case?

In this paper, we give an affirmative answer to the problem stated above. We obtain a sharp upper bound for the second eigenvalue of some general elliptic operators on
closed submanifolds in space forms,
and we can deduce Theorem A and Theorem B as corollaries of our theorems.

Let $M$ be an $n$-dimensional  closed orientable submanifold in an $(n+p)$-dimensional space form $\R$ and $\{e_1,\cdots, e_{n+p}\}$ be an orthonormal frame of $\R$ such that $\{e_1,\cdots, e_{n}\}$ are tangent to $M$, $\{e_{n+1},\cdots, e_{n+p}\}$ are normal to $M$. Denote $A=(h_{ij}^\alpha)$ the second fundamental form of $M$ in $\R$.
Under the frame defined above, generally, given a symmetric, divergence-free $(0,2)$-tensor $T=(T_{ij})$ on $M$ ($T$ can also be regarded as a $(1,1)$-tensor), we define a differential operator $L_T$ as follows:
\begin{equation}\label{LT}
 L_Tf=-\mbox{div}(T\nabla f)=-\sum_{i,j}T_{ij}f_{ij}, \quad f\in C^\infty(M),
\end{equation}
where $f_{ij}$ is the second covariant derivative of $f$.
It is not hard to show that $L$ is self-adjoint since $T$ is divergence-free. $L_T$ is elliptic if and only if $T$ is positive definite.   When $L_T$ is elliptic,  the first eigenvalue of $L_T$ is $0$  obviously and the corresponding eigenfunctions are nonzero constant functions. The minus symbol in the definition \eqref{LT} is to ensure  that the operator is positive-definite. In this case,  $\lambda$ is an  eigenvalue of $L_r$ if there exists a non-zero function $u$ satisfying $L_ru=\lambda u$.
We define a normal vector field $H_T$ associated with $T$ by
\begin{equation}\label{DEFHT}
H_T=\sum_{i=1}^nA(Te_i,e_i)=\sum_{\alpha, i, j}h_{ij}^{\alpha}T_{ij}e_{\alpha}.
\end{equation}
$M$ is called \emph{$T$-minimal} in  $\R$ if $H_T\equiv0$.  The concept \emph{$T$-minimal} is a natural generalization of   \emph{minimal} and  \emph{$r$-minimal} (see Section 2.3).
We obtain the following sharp estimate for the upper bound of the second eigenvalue of the operator $L_T$:
\begin{thm}\label{thm1.1}
	Let $M$ be an $n(>2)$-dimensional  closed orientable submanifold in an $(n+p)$-dimensional space form $\R$. Let $T$ be a symmetric, divergence-free $(1,1)$-tensor on $M$. Assume that $T$ is positive definite, and $T'=(\tr T) I-2T$ is semi-positive definite. Then we have the following sharp estimate for the second eigenvalue of $L_{T}$:
	\begin{equation}\begin{aligned}\label{eq_thm1}
	\la_2^{L_{T}}\leq\frac{1}{V(M)}\int_M\left(c\tr T+\frac{|H_T|^2}{\tr T}\right)\,dv_M.
	\end{aligned}\end{equation}
	The equality  in \eqref{eq_thm1} holds if the following two conditions hold: $(1)$ $\tr T$ is constant; $(2)$ $M$ is $T$-minimal in a geodesic sphere $\Sigma_c$ of $\R$, where the geodesic radius $r_c$ of $\Sigma_c$ is given by \begin{equation*}
	r_0=\left(\frac{\tr T}{\lambda_2^{L_T}}\right)^{1/2},\quad r_1=\arcsin r_0,\quad r_{-1}=\arsinh r_0.
	\end{equation*}
	Moreover, if $T'$ is positive definite and $H_T$ is not identically zero, then the equality in \eqref{eq_thm1} holds  if and only if both the two conditions $(1)$ and $(2)$ hold.
\end{thm}

\begin{rem}
(1) Theorem \ref{thm1.1} can be regarded as an extension of the ``Reilly inequality'' \eqref{Rei}. We also generalize Theorem \ref{thm1.1} to Schr\"{o}dinger-type operators, see Theorem \ref{thm3.1} for the details, which can be regarded as an extension of the  inequality \eqref{SI}.

(2)
We note that the condition ``$H_T$ is not identically zero'' is not essential and it is only used in the case $c=1$. This condition is used to exclude the case that $M$ is $T$-minimal in $\mathbb{S}^{n+p}(1)$ but $M$ is not contained in any totally geodesic $\mathbb{S}^{n+p-1}(1)$ of $\mathbb{S}^{n+p}(1)$.   In fact, if $H_T$ is identically zero, $T'$ is positive definite and the equality in \eqref{eq_thm1} holds,  then it is not hard to prove that the position vector of the immersion $x:M\to \mathbb{S}^{n+p}(1)$ satisfies that $L_T x=\la_2^{L_{T}} x$, then using an analogous argument to that in Proposition \ref{lem5.1}, we get that $M$ is $T$-minimal in $\mathbb{S}^{n+p}(1)$ with $\tr T=\la_2^{L_{T}}$.
\end{rem}

There are lots of natural tensors that satisfy the assumptions in Theorem \ref{thm1.1}. For instance, when $T=T_r$, the $r$-th Newton transformation arising from the shape operator $A$ of a hypersurface in a space form,  $L_T$ is just the well-known $L_r$ operator. In higher codimension case, i.e., $p>1$, when $r$ is even, we can also define a $(0,2)$-tensor $T_r$ and an $L_r$ operator on the submanifolds, which can be regarded as a natural generalization of the $T_r$ operator for hypersurfaces in space forms, see Section 2.3 for the detailed definitions of $T_r$, $L_r$ $H_r$ and $\mathbf{H}_{r+1}$. In Section 4,  by applying Theorem \ref{thm1.1} to the operators $L_r$ for the case $p=1$ or the case $p>1$ and $r$ is even, we obtain

\begin{thm}\label{cor3.15}
	Let $M$ be an $n(\geq2)$-dimensional  closed orientable submanifold in an $(n+p)$-dimensional space form $\R$. For each $r\in\{0,\cdots, n-2\}$ (when $p>1$ we assume $r=2k$ is an even integer),  assume that $T_{r}$ is positive definite and  $(\sum_{k,\alpha}T_{r-1\, kj}^\alpha h_{ki}^\alpha)$ is semi-positive definite.
	Then $L_r$ is elliptic and
	\begin{equation}\begin{aligned}
	\la_2^{L_r}\leq\frac{(n-r)\binom{n}{r}}{V(M)}\int_M\frac{cH_r^2+|\mathbf{H}_{r+1}|^2}{H_r}\,dv_M.\label{eq3.47}
	\end{aligned}\end{equation}
	The equality in \eqref{eq3.47} holds if both the following two conditions hold.
	
	(1) $M$ has constant $r$-mean curvature;
	
	(2) $M$ is $r$-minimal in a geodesic sphere $\Sigma_c$ of $\R$, where the geodesic radius $r_c$ of $\Sigma_c$ is given by
	\begin{equation*}
	r_0=\left(\frac{(n-r)S_r}{\lambda_2^{L_r}}\right)^{1/2},\quad r_1=\arcsin r_0,\quad r_{-1}=\arsinh r_0.
	\end{equation*}
	When $r<n-2$ and $\mathbf{H}_{r+1}$ is not identically zero,  (1) and (2) are also the necessary conditions for the equality.
\end{thm}

\begin{rem}\label{rem3.18}
	 (1)  When $p=1$, by using \eqref{eq2.26}, we have $$(\sum_{k,\alpha}T_{r-1\, kj}^\alpha h_{ki}^\alpha)=(\sum_{k,\alpha}T_{r-1\, kj} h_{ki})=\sum_{k=1}^{r}(-1)^{k-1}S_{r-k}A^{k},$$ where $A=(h_{ij}).$ (cf. \cite{BC97, Gro00}). If $H_{r+1}>0$ and $M$ is convex, then  $T_r$ satisfies the conditions in Theorem \ref{cor3.15}, we can get Theorem A as a corollary of Theorem \ref{cor3.15}.
	
	(2) When $r$ is even,  from Lemma \ref{lem2.3}  in Appendix A, we know that $(\sum_{k,\alpha}T_{r-1\, kj}^\alpha h_{ki}^\alpha)$ is an intrinsic tensor on $M$ corresponding to the Lovelock curvatures. In particular, when $r=2$, we have
	\begin{equation}\label{eq3.50}
	(\sum_{k,\alpha}T_{r-1\, kj}^\alpha h_{ki}^\alpha)=\Ric-(n-1)c I,
	\end{equation}
	in this case, the condition ``$(\sum_{k,\alpha}T_{r-1\, kj}^\alpha h_{ki}^\alpha)$ is semi-positive definite'' is equivalent to $\Ric\geq (n-1)c I$. We give some examples in Section 4 which satisfy the conditions in Theorem \ref{cor3.15} and the equality in \eqref{eq3.47}.
\end{rem}

Note that when $p>1$, one cannot define the operator $L_r$ if $r$ is odd. However, if the mean curvature vector is nowhere zero, we choose $e_{n+1}=\frac{\mathbf{H}}{H}$, where $H=|\mathbf{H}|$ is the mean curvature, $e_{n+1}$ is usually called the principal normal. We can define a tensor $T$ by
\begin{equation}\label{T1}T_{ij}=nH\delta_{ij}-h_{ij}^{n+1},
\end{equation}
where $h^{n+1}_{ij}$ denotes the component of the second fundamental form in the direction of the principal normal. Then its corresponding  operator is given by
\begin{equation}\label{L1}
L f=-\sum_{i,j}(nH\delta_{ij}-h_{ij}^{n+1})f_{ij}, ~\forall f\in C^{\infty}(M).
\end{equation}
This operator is a natural generalization of the operator $L_1$ on hypersurfaces in space forms.
In \cite{GL13}, X. Guo and H. Li used the properties of $L$ to prove a rigidity theorem for submanifolds with constant scalar curvature and parallel normalized mean curvature vector field in a unit sphere. In Section 5,  by applying Theorem \ref{thm1.1} to the tensor $T$ defined by \eqref{T1} and the  operator $L$ defined by \eqref{L1}, we get an optimal upper bound for the second eigenvalue of $L$ by assuming that $n\geq4$ and $H_2>0.$ Moreover, we prove that the equality holds if and only if $M$ is a minimal submanifold in a geodesic sphere of $\R$. We note that the  conclusion that $M$ is minimal is  stronger  than
 the conclusion that $M$ is $T$-minimal.  More precisely, we prove the following result.
\begin{thm}\label{thm1.2}
	Let $M$ be an $n$-dimensional  closed orientable submanifold in an $(n+p)$-dimensional space form $\R$. We assume that $n\geq4,~H_2>0.$ Then the mean curvature vector $\mathbf{H}\neq 0$. We choose an orthonormal frame $\{e_1,\cdots, e_{n+p}\}$ as above.  Then the operator $L$ defined by \eqref{L1} is elliptic, and the second eigenvalue of $L$ satisfies
	\begin{equation}\label{eqthm}
	\la_2^{L}\leq\frac{n(n-1)}{V(M)}\int_M\bigg[cH+\frac{1}{H}\Big(H_2+\frac{|\tau|^2}{n(n-1)}\Big)^2+\frac{1}{n^2(n-1)^2H}\sum\limits_{\alpha\geq n+2}\Big(\sum\limits_{i,j}h_{ij}^{n+1}h_{ij}^{\alpha}\Big)^2\bigg]\,dv_M,
	\end{equation}
where $|\tau|^2=\displaystyle\sum_{\alpha\neq n+1}(h_{ij}^{\alpha})^2$.
	
Moreover, the equality in \eqref{eqthm} holds  if and only if $M$ is a minimal submanifold in a geodesic sphere $\Sigma_c$ of $\R$, where the geodesic radius $r_c$ of $\Sigma_c$ is given by  $$r_0=\left(\frac{n}{\lambda_2^\Delta}\right)^{1/2}, ~r_1=\arcsin r_0,~ r_{-1}=\arsinh r_0.$$
\end{thm}

\begin{rem}
 When $p=1$,  we obtain the same estimate for the case $r=1$  in Theorem A without the convexity assumption.
By adapting the proof of Theorem \ref{thm1.2} and using the Newton-Maclaurin inequalities, one can prove that Theorem B is true for $n\geq4$, which gives a partial answer to the problem mentioned in Remark 4.3 of \cite{LW12}.
\end{rem}

The paper is organized as follows. In Section 2, we show some basic formulas for a submanifold $(M, g_M)$ in a Riemannian manifold $(N^{n+p}, g_N)$, including the relations between some geometric quantities of $M$ as a submanifold in  $(N^{n+p}, g_N)$ and their corresponding quantities of $M$ as a submanifold in $(N^{n+p}, \tilde{g}_N)$, where $\tilde{g}_N$ is conformal to $g_N$. These relations will be used in the proof of our main theorem.
In Section 3, we prove Theorem \ref{thm1.1} and generalize Theorem \ref{thm1.1} to a Schr{\"o}dinger-type operator $L_T+q$. In Section 4, we apply Theorem \ref{thm1.1} to the $L_r$ operator for the case $p=1$ or the case $p>1$ and $r$ is even, we also give some examples. In Section 5, we prove Theorem \ref{thm1.2}.

Our main contributions in this paper are in two aspects: first, we can deal with a class of very general elliptic operators for submanifolds in  space forms for any codimension. The key ingredient is to find the relations between some geometric quantities associated with the elliptic operator,  which is presented in Proposition \ref{prop3.2}; second, we prove that the inequality is sharp and give the sufficient and necessary conditions when the upper bound is attained. This part is more difficult as we are in the arbitrary codimension case. If the codimension is $1$, the sufficient and necessary conditions can be proved easily. To overcome the difficulties coming from the higher codimension case, we write down explicitly the conformal transformation used in the proof and analyze the equality case very carefully, and the key steps are contained in Lemmas \ref{lem5-1}, \ref{lem5-0} and \ref{lem5--1}. The most important observation in these lemmas is that we find that the the component $\Phi^0$ is $0$ on $M$. Moreover, in Theorem  \ref{thm1.2}, we prove that $M$ is  minimal in a geodesic sphere of $\R$, which is stronger than the conclusion $T$-minimal, so we cannot get Theorem \ref{thm1.2} by applying Theorem \ref{thm1.1} directly, and we need to do some further analysis for the equality.

\textbf{Acknowledgment}:
The first author was partially supported by NSFC Grant No. 11601426 and the Fundamental Research Funds for the Central Universities Grant No. 3102014JCQ01071.
The second author was partially supported by NSFC  Grant No. 11571185.
  The authors would like to thank Professor Haizhong Li for his constant support through the years. This work is partially done while the first author is visiting UCSB. He would like to thank UCSB math department for the hospitality, and he also would like to acknowledge financial support from China Scholarship Council and Top International University Visiting Program for Outstanding Young scholars of Northwestern Polytechnical University.

\section{Preliminaries and notations}
In this section, we give the relations between some geometric quantities of $M$ as a submanifold in  $(N^{n+p}, g_N)$ and their corresponding quantities of $M$ as a submanifold in $(N^{n+p}, \tilde{g}_N)$, where $\tilde{g}_N$ is conformal to $g_N$. Although the relations are well-known in the literature (cf. \cite{Che73a, Che74}), we give a brief proof of these relations for the reader's convenience as the relations  will be used in the proof of Theorem \ref{thm1.1}. Next, we also recall some basic formulas for submanifolds in space forms. At last, we recall  the definitions and properties of  the $L_r$ operator and the $r$-mean curvature.

We use the following convention on the ranges of indices except special declaration:
\begin{equation*}
1\leq i, j, k, \ldots \leq n;\quad
n+1\leq\alpha, \beta, \gamma, \ldots \leq n+p;\quad
1\leq A, B, C, \ldots \leq n+p.
\end{equation*}
\subsection{Conformal relations}
Let $M$ be an $n$-dimensional submanifold in an $(n+p)$-dimensional Riemannian manifold $(N^{n+p}, g_N)$, where $g_N$ is the metric of $N$. We denote the immersion from $M$ to $N$ by $x$, then $M$ has an induced metric $g_M=x^{\ast}g_N.$ We denote the Levi-Civita connections on $M$ and $N$ by $\nabla$ and $\bar{\nabla}$ respectively. Let $\{e_A\}_{A=1}^{n+p}$ be an orthonormal frame of $(N,g_N)$, where $\{e_i\}_{i=1}^n$ are tangent to $M$ and $\{e_{\alpha}\}_{\alpha=n+1}^{n+p}$ are normal to $M$; Let $\{\omega_A\}_{A=1}^{n+p}$ be the dual coframe of $\{e_A\}_{A=1}^{n+p}$ . Then the structure equations of $(N,g_N)$ are (see \cite{Chern1968}):
\begin{equation}
\left\{\begin{aligned}\label{eq21}
& d\omega_{A}=\sum_{B}\omega_{AB}\wedge\omega_{B}, \\
&\omega_{AB}+\omega_{BA}=0,
\end{aligned}\right.
\end{equation}
where $\{\omega_{AB}\}$ are the connection forms of $(N,g_N)$.

Denote $x^{\ast}\omega_{A}=\theta_{A}, x^{\ast}\omega_{AB}=\theta_{AB}$, then restricted to $(M,g_M)$, we have (see \cite{Chern1968})
\begin{equation}\label{eq23}
\theta_{\alpha}=0, \quad\theta_{i\alpha}=\sum_j h_{ij}^{\alpha}\theta_{j}.
\end{equation}
and
\begin{equation}
\left\{\begin{aligned}\label{eq22}
&d\theta_{i}=\sum_{j}\theta_{ij}\wedge\theta_{j}, \quad\theta_{ij}+\theta_{ji}=0,\\
& d\theta_{ij}-\sum_k\theta_{ik}\wedge\theta_{kj}=-\frac{1}{2}\sum_{k,l}R_{ijkl}\,\theta_k\wedge\theta_l,
\end{aligned}
\right.
\end{equation}
where $R_{ijkl}$ are components of the curvature tensor of $(M, g_M)$ and $h_{ij}^{\alpha}$ are components of the second fundamental form of $(M, g_M)$ in $(N, g_N)$.

Now we assume that $N$ is equipped with a new metric $\tg_N=e^{2\rho}g_N$ which is conformal to $g_N$, where $\rho\in C^{\infty}(N)$. Then $\{\et_A=e^{-\rho}e_A\}$ is an orthonormal frame of $(N,\tg_N)$, and $\{\tom_A=e^{\rho}\omega_A\}$ is the dual coframe of  $\{\et_A\}$. We denote the Levi-Civita connections on $(N,\tg_N)$ by  $\tilde{\nabla}$. The structure equations of $(N,\tg_N)$ are given by

\begin{equation}
\left\{\begin{aligned}\label{eq25}
 &d\tom_{A}=\sum\limits_{B}\tom_{AB}\wedge\tom_{B},\\
 &\tom_{AB}+\tom_{BA}=0,
\end{aligned}\right.
\end{equation}
where $\{\tom_{AB}\}$ are the connection forms of $(N,\tg_N)$.

Given a smooth function $F$ on $(N,g_N)$, its gradient is given by (see \cite{Chern1968})
\begin{equation*}
dF=\sum_{A=1}^{n+p}F_A\omega_A=\sum_{A=1}^{n+p}\bar{\nabla}_{A}F\;\omega_A.
\end{equation*}
The second covariant derivative are given by
\begin{equation*}
	\sum_{B}F_{AB}=dF_A+\sum_BF_B\omega_{BA}.
\end{equation*}
On the other hand, under the metric $\tg_N$,
\begin{equation*}
dF=\sum_{A=1}^{n+p}\tilde{F}_A\tom_A=\sum_{A=1}^{n+p}\tn_{A}F\;\tom_A,
\end{equation*}
so we have the following relation
\begin{equation}\label{eq27}
\tn_{A}F=e^{-\rho}\bar{\nabla}_{A}F,~~\forall F\in C^{\infty}(N).
\end{equation}

From (\ref{eq21}) and (\ref{eq25}), we derive that
\begin{equation}
\begin{aligned}\label{eq26}
\tom_{AB}=\omega_{AB}+\rho_A\omega_B-\rho_B\omega_A,
\end{aligned}
\end{equation}
where $\rho_A$ means the covariant derivative of $\rho$ with respect to $e_A$.

We denote $\tg_M=x^{\ast}\tg_N, x^{\ast}\tom_{A}=\tt_{A}, x^{\ast}\tom_{AB}=\tt_{AB}$, then restricted to $(M,\tg_M)$, we have
\begin{equation}\label{eq29}\tt_{\alpha}=0, \quad\tt_{i\alpha}=\sum_j \th_{ij}^{\alpha}\tt_{j},
\end{equation}
and
\begin{equation}
\left\{
\begin{aligned}\label{eq28}
&\displaystyle d\tt_{i}=\sum_{j}\tt_{ij}\wedge\tt_{j}, \quad\tt_{ij}+\tt_{ji}=0, \\
\displaystyle &d\tt_{ij}-\sum_k\tt_{ik}\wedge\tt_{kj}=-\frac{1}{2}\sum_{k,l}\tR_{ijkl}\,\tt_k\wedge\tt_l,
\end{aligned}
\right.\end{equation}
where $\tR_{ijkl}$ are components of the curvature tensor of $(M, \tg_M)$ and $\th_{ij}^{\alpha}$ are components of the second fundamental form of $(M, \tg_M)$ in $(N, \tg_N)$.

By pulling back (\ref{eq26}) to $M$ by $x$ and using (\ref{eq23}) and (\ref{eq29}), we obtain the following relation.
\begin{equation}\label{eq208}
\th_{ij}^{\alpha}=e^{-\rho}(h_{ij}^{\alpha}-\rho_{\alpha}\delta_{ij}), \quad\tH^{\alpha}=e^{-\rho}(H^{\alpha}-\rho_{\alpha}).
\end{equation}
Combining (\ref{eq22}), (\ref{eq26}), (\ref{eq27}) and (\ref{eq28}), we can obtain the following relation:
\begin{equation}\label{eq209}
\begin{aligned}
e^{2\rho}\tR_{ijkl}=&R_{ijkl}-(\rho_{ik}\delta_{jl}+\rho_{jl}\delta_{ik}-\rho_{il}\delta_{jk}-\rho_{jk}\delta_{il})\\
&+(\rho_i\rho_k\delta_{jl}+\rho_j\rho_l\delta_{ik}-\rho_j\rho_k\delta_{il}-\rho_i\rho_l\delta_{jk})\\
&-(\sum_m\rho_{m}^2)(\delta_{ik}\delta_{jl}-\delta_{il}\delta_{jk}),
\end{aligned}
\end{equation}
\begin{equation}\label{eq210}
e^{2\rho}\tR_{ij}=R_{ij}+(n-2)\big(\rho_i\rho_j-\rho_{ij}-|\nabla \rho|^2\delta_{ij}\big)-(\Delta\rho)\delta_{ij},
\end{equation}
where $R_{ij}$ (or $\tR_{ij}$ resp.) are components of Ricci curvature with respect to $g_M$ (or $\tg_M$ resp.).

\subsection{Basic formulas for submanifolds in  space forms}
From now on, we set $N=\R$, and denote $x$ the immersion from $M$ to $\R$. Using the previous notations, when restricted to $M$, we have the following structure equations of $M$ (see \cites{CL07,Chern1968}):
\begin{align*}
dx=&\sum_{i}\theta_ie_i,\quad de_i=\sum_{j}\theta_{ij}e_j+\sum_{j}h_{ij}^{\alpha}\theta_{j}e_{\alpha}-c\theta_ix,\quad
de_{\alpha}=-\sum_{i,j}h_{ij}^{\alpha}\theta_{j}e_i+\sum_{\beta}\theta_{\alpha\beta}e_{\beta},
\end{align*}
from which we derive that  (cf. \cite{Tak66,CL07})
\begin{equation}\label{eq_d2x}
x_i=e_i,\quad x_{ij}=\sum_{\alpha}h_{ij}^{\alpha}e_{\alpha}-c\delta_{ij}x.
\end{equation}

The Gauss equations are given by (see \cites{CL07,Chern1968})
\begin{align}\label{gauss}
R_{ijkl}&=(\delta_{ik}\delta_{jl}-\delta_{il}\delta_{jk})c+\sum_\alpha( h_{ik}^{\alpha}h_{jl}^{\alpha}-h_{il}^{\alpha}h_{jk}^{\alpha}),\\
\label{gauss2}
R_{ik}&=(n-1)c\delta_{ik}+nH^{\alpha}h_{ik}^{\alpha}-\sum_{\alpha,j}h_{ij}^{\alpha}h_{jk}^{\alpha},
\\\label{gauss3}
R&=n(n-1)c+n^2H^2-S,
\end{align}
where $R$ is the scalar curvature of $M$, $S=\sum\limits_{\alpha,i,j}(h_{ij}^{\alpha})^2$ is the
norm square of the second fundamental form, $\mathbf{H}=\sum\limits_\alpha H^{\alpha}e_\alpha=\frac{1}{n}
\sum\limits_\alpha(\sum\limits_i h_{ii}^{\alpha})e_\alpha$ is the mean curvature  vector of $M$.

The Codazzi
equations are given by (see \cites{CL07,Chern1968})
\begin{equation}\label{codazzi}
h_{ijk}^{\alpha}=h_{ikj}^{\alpha},
\end{equation}
where the covariant derivative of $h_{ij}^{\alpha}$ is defined by
\begin{equation*}
\sum_kh_{ijk}^{\alpha}\theta_k=dh_{ij}^{\alpha}+\sum_kh_{kj}^{\alpha}\theta_{ki}+\sum_kh_{ik}^{\alpha}\theta_{kj}+\sum_\beta h_{ij}^{\beta}\theta_{\beta\alpha}.
\end{equation*}

The  gradient and
Hessian of $f\in C^\infty(M)$ are given by
\begin{equation*}
df=\sum_{i=1}^{n}f_i\theta_{i},\quad
\sum_{j=1}^{n}f_{ij}\theta_j=df_i+\sum_{j=1}^{n}f_j\theta_{ji}.
\end{equation*}

\subsection{Newton transformations $T_r$, the $L_r$ operator and  the $r$-mean curvature.}
By convention, we set $H_0=S_0=1$ and $T_0=I$.
We denote $A_{ij}=A(e_i,e_j)=\sum_\alpha h_{ij}^\alpha e_\alpha$ and define $(0,2)$-tensor $T_r$ for $r\in {1,\cdots, n}$ as follows (cf. \cites{Rei77,Gro02,CL07}).
\begin{align*}
	T_r=&\frac{1}{r!}\sum_{\substack{i_1\ldots i_ri\\j_1\ldots j_rj}}\delta^{i_1\ldots i_ri}_{j_1\ldots j_rj}\langle A_{i_1j_1},A_{i_2j_2}\rangle\cdots\langle A_{i_{r-1}j_{r-1}},A_{i_rj_r}\rangle \theta_i\otimes\theta_j\\
	=&\sum_{i,j}T_{r\,ij}\theta_i\otimes\theta_j, \quad \mbox{if $r$ is even.}\\
T_r=&\frac{1}{r!}\sum_{\substack{i_1\ldots i_ri\\j_1\ldots j_rj}}\delta^{i_1\ldots i_ri}_{j_1\ldots j_rj}\langle A_{i_1j_1},A_{i_2j_2}\rangle\cdots\langle A_{i_{r-2}j_{r-2}},A_{i_{r-1}j_{r-1}}\rangle A_{i_rj_r} \theta_i\otimes\theta_j\\
=&\sum_{i,j,\alpha}T_{r\,ij}^\alpha e_\alpha\theta_i\otimes\theta_j,\quad \mbox{if $r$ is odd.}
\end{align*}
Here $\delta^{i_1\ldots i_ri}_{j_1\ldots j_rj}$ is the generalized Kronecker delta.

When $r$ is even, $T_r$ is a symmetric and divergence-free $(0,2)$-tensor, and
one can define a differential operator $L_r$ associated to $T_r$ by (see \cites{CY77,CL07})
\begin{equation}\label{Lr}
L_r f=-\sum_{i,j}T_{r\,ij}f_{ij}=-\mbox{div}(T_r\nabla f),~\forall~f\in C^{\infty}(M).
\end{equation}

When $r$ is odd,
 $T_r$ is a symmetric and divergence-free normal-vectored value $(0,2)$-tensor and one can define a differential operator $L_r$ by
 \begin{equation*}
L_r f=-\sum_{i,j}T^\alpha_{r\,ij}f_{ij}e_{\alpha},~\forall~f\in C^{\infty}(M),
 \end{equation*}
 which maps smooth functions to the sections of the normal bundle of $M$  (cf. \cite{CL07}).

When $r$ is even, the
$r$-th mean curvature function  $S_r$ and $(r+1)$-th
mean curvature vector field $\mathbf{S}_{r+1}$ are defined as follows:
\begin{equation}\label{eq2.23}
S_r=\frac{1}{r}\sum_{\alpha, i, j}T_{r-1\, ij}^{\alpha}h_{ij}^{\alpha}=\binom{n}{r}H_r,\quad
\mathbf{S}_{r+1}=\frac{1}{r+1}\sum_{\alpha, i, j}T_{r\, ij}h_{ij}^{\alpha}e_{\alpha}=\binom{n}{r+1}\mathbf{H}_{r+1}.
\end{equation}
A submanifold $M$ is called  \emph{$r$-minimal} if its $(r + 1)$-th mean curvature vector $\mathbf{S}_{r+1}$ vanishes on $M$ (see Definition 1.1 in \cite{CL07}).

When the codimension $p=1$, we denote $h_{ij}=h_{ij}^{n+1}$ and replace $T_r$ by the tensor $T_r = \langle T_r(\cdot, \cdot), e_{n+1}\rangle$ for odd $r$. Then
\begin{equation*}
	S_r=\sum_{i_1<\cdots<i_r}k_{i_1}\cdots k_{i_r}
\end{equation*}
is the $r$-th elementary symmetric polynomial of principal curvatures $\{k_1,\cdots, k_n\}$.

We recall some basic facts for later use.
\begin{lem}[see Lemma 3.3, Lemma 3.4 of \cite{CL07} and Lemma 2.1 of \cite{BC97}]\label{lem2.1}
	For any integer $r\in \{0,\ldots, n-1\}$, we have
	\begin{align}
	&\tr(T_r)=(n-r)S_r,\quad \mbox{when $p=1$ or $p>1$ and $r$ is even};\nonumber\\
	&\tr(T_{r\,ij}^\alpha)=\frac{n-r}{r}\sum_{i,j}T_{r-1\,ij}h_{ij}^\alpha, \quad\mbox{for each $\alpha$ when $r$ is odd};\nonumber\\
	&T_{r\,ij}=S_r\delta_{ij}-\sum_{k,\alpha}T_{r-1\,kj}^\alpha h_{ki}^\alpha,\quad \mbox{when $p=1,~r\geq 1$ or $p>1,r\geq 1$ and $r$ is even}.\label{eq2.26}
	\end{align}
\end{lem}

\section{Proof of Theorem \ref{thm1.1}}
In this section, we prove Theorem \ref{thm1.1}. First, to prove the inequality \eqref{eq_thm1}, the first step is using a technique of conformal transformation on a sphere which was introduced by Li and Yau (see \cite{LY82}) and was used by other authors (see \cites{MR86, Per04, Urbano90, ESI00, Gro00, LW12,CW13}). This technique will provide us  good test functions to estimate the second eigenvalue of $L_T$. After choosing the suitable test functions, the key step is to find the relations between some geometric quantities associated with $T$, which is presented in Proposition \ref{prop3.2}. Second, to prove the sufficient and necessary conditions of the equality in \eqref{eq_thm1}, we need to write down explicitly the conformal transformation and analyze the equality carefully. We will discuss the three cases ($c=1,0,-1$) separately. Finally, we generalize Theorem \ref{thm1.1} to Schr{\"o}dinger-type operators.

\subsection{The inequality in \eqref{eq_thm1}}

 By using the technique in Li-Yau \cite{LY82}, we have the following lemma.
\begin{lem}[see \cites{LY82,ESI00}]\label{lem3.1}
Let $M$ be an $n$-dimensional closed orientable submanifold in an $(n+p)$-dimensional space form $\R$. Then there exists a regular conformal map $\Gamma: \R\to \S\subset\mathbb{R}^{n+p+1}$ such that  the immersion $\Phi=\Gamma\circ x=(\Phi^1, \cdots, \Phi^{n+p+1})$ satisfies that
\begin{equation*}
\int_M \Phi^A\,dv_M=0,~A=1,\ldots,n+p+1.
\end{equation*}
\end{lem}

\begin{rem}
We note that the immersion $\Phi$ in Lemma \ref{lem3.1} can be written down explicitly, see Section 3.3 for the details.
\end{rem}

Now we set $N=\R$ in Section 2.1, $g_N=h_c, \tg_N=\Gamma^{\ast}h_1=e^{2\rho}g_N$, then $g_M=x^{\ast}h_c$, $\tg_M=(\Gamma\circ x)^{\ast}h_1.$ Here $h_c$ is the standard metric on $\R$.
In order to obtain the inequality \eqref{eq_thm1}, we first prove the following key proposition.
\begin{prop}\label{prop3.2}
Let $M$ be an $n(> 2)$-dimensional  closed orientable submanifold in an $(n+p)$-dimensional space form $\R$. Let $T$ be a symmetric, divergence-free $(1,1)$-tensor on $M$. Then we have the following relation.
	\begin{equation*}
	e^{2\rho}\tr T=c\tr T+2L_T\rho-\tr T|(\bar{\nabla}\rho)^\bot|^2+2\langle H_T, (\bar{\nabla}\rho)^\bot\rangle-T'(\nabla\rho,\nabla\rho),
\end{equation*}
	where $T'=(\tr T)I-2T$, $\nabla\rho=\sum_i\rho_ie_i, (\bar{\nabla}\rho)^\bot=\sum_\alpha\rho_\alpha e_\alpha=\bar{\nabla}\rho-\nabla\rho.$
\end{prop}
\begin{proof}
From \eqref{gauss2}, we have the Gauss equations for the immersion $x$ and the immersion $\Phi=\Gamma\circ x$ respectively:
	\begin{align}
	R_{ij}=&(n-1)c\delta_{ij}+\sum_{\alpha}nH^{\alpha}h_{ij}^{\alpha}-\sum_{k,\alpha}h_{ik}^{\alpha}h_{kj}^{\alpha},\label{G1}\\
	\tR_{ij}=&(n-1)\delta_{ij}+\sum_{\alpha}n\tH^{\alpha}\th_{ij}^{\alpha}-\sum_{k,\alpha}\th_{ik}^{\alpha}\th_{kj}^{\alpha}.\label{G2}
\end{align}
	From (\ref{eq208}), (\ref{eq210}), (\ref{G1}) and (\ref{G2}), we obtain
	\begin{align}
	&(n-2)\big(\rho_i\rho_j-\rho_{ij}-|\nabla \rho|^2\delta_{ij}\big)-(\Delta \rho)\delta_{ij}
	=e^{2\rho}\tR_{ij}-R_{ij}\nonumber\\
	 =&(n-1)(e^{2\rho}-c+\sum_{\alpha}\rho_{\alpha}^2)\delta_{ij}-(n-2)\sum_{\alpha}\rho_{\alpha}h_{ij}^{\alpha}-\sum_{\alpha}nH^\alpha\rho_{\alpha}\delta_{ij}.\label{eq3.7}
   \end{align}
Contracting \eqref{eq3.7} with $\delta_{ij}$ and $T_{ij}$ respectively, we obtain
\begin{equation}\label{deltarho}
	-2\Delta \rho=(n-2)|\nabla \rho|^2+n(e^{2\rho}-c+\sum_{\alpha}\rho_{\alpha}^2)-2\sum_{\alpha}nH^\alpha\rho_{\alpha},
\end{equation}
	\begin{align*}
	&(n-1)\tr T(e^{2\rho}-c+\sum_{\alpha}\rho_{\alpha}^2)\\
	=&(n-2)\sum_{\alpha, i, j}\rho_{\alpha}h_{ij}^{\alpha}T_{ij}+\sum_{\alpha}nH^\alpha\rho_{\alpha}\tr T+(n-2)\sum_{i, j}\rho_i\rho_jT_{ij}\\
	&-(n-2)\sum_{i, j}\rho_{ij}T_{ij}-(n-2)\tr T |\nabla\rho|^2-\tr T(\Delta\rho)\\
	=&(n-2)\sum_{\alpha, i, j}\rho_{\alpha}h_{ij}^{\alpha}T_{ij}+\sum_{\alpha}nH^\alpha\rho_{\alpha}\tr T+(n-2)\sum_{i, j}\rho_i\rho_jT_{ij}\\
	&+(n-2)L_T\rho-(n-2)\tr T |\nabla\rho|^2\\
	&+\frac{\tr T}{2}\Big((n-2)|\nabla\rho|^2+n(e^{2\rho}-c+\sum_{\alpha}\rho_{\alpha}^2)-2\sum_{\alpha}nH^\alpha\rho_{\alpha}\Big).
	\end{align*}
	Hence we get
	\begin{align*}
	&\frac{n-2}{2}\tr T(e^{2\rho}-c+\sum_{\alpha}\rho_{\alpha}^2)\\
	=&(n-2)\sum_{\alpha, i, j}\rho_{\alpha}h_{ij}^{\alpha}T_{ij}+(n-2)L_T\rho-\frac{(n-2)\tr T}{2}|\nabla\rho|^2+(n-2)\sum_{i,j}\rho_i\rho_jT_{ij}.
\end{align*}
As we assume that $n>2$, we obtain that
\begin{align*}
	\tr T(e^{2\rho}-c+\sum_{\alpha}\rho_{\alpha}^2)
	=2\sum_{\alpha, i, j}\rho_{\alpha}h_{ij}^{\alpha}T_{ij}+2L_T\rho-\tr T|\nabla\rho|^2+2\sum_{i,j}\rho_i\rho_jT_{ij},
\end{align*}
from which we immediately get
\begin{equation*}
	e^{2\rho}\tr T=c\tr T+2L_T\rho-\tr T|(\bar{\nabla}\rho)^\bot|^2+2\langle H_T, (\bar{\nabla}\rho)^\bot\rangle-T'(\nabla\rho,\nabla\rho),
\end{equation*}
	where $T'=(\tr T)I-2T$, $\nabla\rho=\sum_i\rho_ie_i, (\bar{\nabla}\rho)^\bot=\sum_\alpha\rho_\alpha e_\alpha=\bar{\nabla}\rho-\nabla\rho.$
\end{proof}
\begin{rem}\label{remn2}
We note that if $n=2$, $T_{ij}=\delta_{ij}$, then $L_T=-\Delta$ is the Laplacian on $M$, from \eqref{deltarho}, one can see that
Proposition \ref{prop3.2} is still valid for this case.
\end{rem}
\noindent{\bf Proof of the inequality \eqref{eq_thm1}:}
Since we assume that $T$ is positive definite, the operator $L_T$ defined by \eqref{LT} is  elliptic and has a discrete nonnegative spectrum.
The first eigenvalue of $L_T$ is $0$ and the corresponding eigenfunctions are nonzero constant functions. Lemma \ref{lem3.1} implies each coordinate function $\Phi^A$ is $L^2$-orthogonal to the first eigenfunction, then by using the min-max principle, we have
\begin{equation}\begin{aligned}\label{eq31}
\la_2^{L_T}\int_M(\Phi^A)^2\,dv_M\leq\int_M\Phi^AL_T\Phi^A\,dv_M, \quad (1\leq A\leq n+p+1).
\end{aligned}\end{equation}
By using (\ref{eq27}), we have
\begin{equation}\begin{aligned}\label{eq32}
\sum_{A=1}^{n+p+1}\Phi^A_{\,i}\Phi^A_{\,j}=e^{2\rho}\sum_{A=1}^{n+p+1}\tn_i\Phi^A\tn_j\Phi^A=e^{2\rho}\delta_{ij}.
\end{aligned}\end{equation}
Summing up (\ref{eq31}) over $A$ and using (\ref{eq32}) and the fact that $\sum\limits_{A=1}^{n+p+1}(\Phi^A)^2=1$, we obtain
\begin{align}
\la_2^{L_T}\,V(M)&\leq\sum_{A=1}^{n+p+1}\int_M\Phi^AL_T\Phi^A\,dv_M=\sum_{A=1}^{n+p+1}\int_M\sum_{i,j}\Phi^A_{\,i}\Phi^A_{\,j}T_{ij}\,dv_M\nonumber\\
&=\int_M\sum_{i,j}e^{2\rho}\delta_{ij}T_{ij}\,dv_M=\int_Me^{2\rho}\tr T\,dv_M.\label{eigen}
\end{align}
Since $T'$ is semi-positive definite and $L_T$ is self-adjoint,  from \eqref{eigen} and Proposition \ref{prop3.2}, we have
\begin{align}
	\la_2^{L_T}\,V(M)&\leq \int_Me^{2\rho}\tr T\,dv_M\nonumber\\
	&\leq\int_M \left(c\tr T-\tr T\left|(\bar{\nabla}\rho)^\bot -\frac{1}{\tr T}H_T\right|^2+\frac{|H_T|^2}{\tr T}\right)\,dv_M\label{eq33}\\
	&\leq\int_M \left(c\tr T+\frac{|H_T|^2}{\tr T}\right)\,dv_M.\label{eq34}
\end{align}
\qed

\subsection{Sufficient condition for the equality in \eqref{eq_thm1}}

First we give a Takahashi-type result which characterizes the $T$-minimal submanifold in a sphere and can be compared with Theorem \ref{thm1.1}.

\begin{prop}\label{lem5.1}
	Let $x: M\to \mathbb{S}^{n+m}$ be an $n$-dimensional closed orientable submanifold in a sphere $\mathbb{S}^{n+m}$ of constant curvature $c' (>0)$ and $T$ be a symmetric, divergence-free $(1,1)$-tensor on $M$.  Then $M$ is a $T$-minimal submanifold in $\mathbb{S}^{n+m}$ if and only if $L_{T}x= c'(\tr T)x$.  As a corollary, assume that $T$ is positive definite and $\tr T$ is constant, if $M$ is $T$-minimal in the sphere $\mathbb{S}^{n+m}$ of constant curvature $c' (>0)$, then  $\lambda_2^{L_T}\leq c'(\tr T)$.
\end{prop}
\begin{proof}
	The proof is inspired by the famous Takahashi theorem (cf. \cite{Tak66}).
	Let $x$ be the position vector of $M$ in $\mathbb{S}^{n+m}\subset\mathbb{R}^{n+m+1}$, then using \eqref{eq_d2x}, we have
$$L_Tx=-\sum_{i,j}T_{ij}x_{ij}=c'(\tr T) x-H_T,$$ which shows that  $L_{T}x= c'(\tr T)x$ if and only if $H_T=0$, i.e., $M$ is $T$-minimal in $\mathbb{S}^{n+m}$.

When $T$ is positive definite and $\tr T$ is constant, if $M$ is $T$-minimal  in the sphere $\mathbb{S}^{n+m}$ of constant curvature $c' (>0)$,  then we have $L_{T}x= c'(\tr T)x$, which means that $c'(\tr T)$ is a positive eigenvalue of $L_T$ and each coordinate function (if  it is not $0$) is an eigenfunction corresponding to $c'(\tr T)$, so we get $\lambda_2^{L_T}\leq c'(\tr T)$.
\end{proof}

\noindent \textbf{Proof of the sufficient condition for the equality in \eqref{eq_thm1}:}
Assume that  $\tr T$ is constant and  $M$ is $T$-minimal in a geodesic sphere $\Sigma_c$ of $\R$, where the geodesic radius $r_c$ of $\Sigma_c$ is given by
\begin{equation}\label{eq_rad}
r_0=\left(\frac{\tr T}{\lambda_2^{L_T}}\right)^{1/2},\quad r_1=\arcsin r_0,\quad r_{-1}=\arsinh r_0.
\end{equation}
To show that the equality in \eqref{eq_thm1} holds,
we need to compute the right hand side of \eqref{eq_thm1}.

When $M$ is $T$-minimal in a geodesic sphere $\Sigma_c$ of constant curvature $c'$ of $\R$ ($c'\geq 1$ for $c=1$ and $c'>0$ for $c=0,-1$), we choose the unit normal vector of the immersion from $\Sigma_c$ to $\R$ to be $e_{n+1}$, and let $\{e_{n+2}, \cdots, e_{n+p}\}$ be the normal frame of the immersion from $M$ to $\Sigma_c$ (When $p= 1$, $M=\Sigma_c$ and there is only one normal vector $e_{n+1}$). We choose an orthonormal tangent frame $\{e_1,\cdots, e_n\}$ on $M$, then we have $h_{ij}^{n+1}=k\delta_{ij}$, where $k=\sqrt{c'-c}\geq 0$ is the principal curvature of $\Sigma_c$ in $\R$ (the expression of $k$ can be obtained by using Gauss equation, and we can always choose $e_{n+1}$ such that $k\geq 0$). Since $M$ is $T$-minimal in $\Sigma_c$, we have  $\sum_{\beta=n+2}^{n+p}T_{ij}h_{ij}^{\beta}e_{\beta}=0$, so we get that
$H_T=\sum_{i,j}T_{ij}h_{ij}^{n+1}e_{n+1}=(\tr T)ke_{n+1}$. Since $\tr T$ is constant, we immediately get that the right side of \eqref{eq_thm1} is $(c+k^2)\tr T=c'\tr T$.

On the other hand, a basic fact is that the geodesic radius $r_c$ and the curvature $c'$ of $\Sigma_c$ have the following relation \begin{equation}\label{eq_rad1}c'=\frac{1}{r_0^2}=\frac{1}{\sin^2r_1}=\frac{1}{\sinh^2r_{-1}}.\end{equation}
Therefore, by using \eqref{eq_rad}, we get $\lambda_2^{L_T}=c'\tr T$.  This completes the proof of the sufficient condition for the equality in \eqref{eq_thm1}.\qed

\subsection{Necessary condition for the equality in \eqref{eq_thm1}}

In this subsection, we discuss the necessary condition for the equality in \eqref{eq_thm1}.  For convenience, we will denote $N=n+p$, and set $x=(\tilde{x}, x^0)$ for any vector $x\in\mathbb{R}^{N+1}$, where $\tilde{x}=(x^1,\cdots, x^N)\in\mathbb{R}^{N},x^0\in\mathbb{R}^{1}$.

First of all, note that if equality holds in \eqref{eq_thm1}, then we have the following necessary conditions:

(N1) $L_T\Phi^A=\lambda_2^{L_T}\Phi^A$ on $M$ for all $A (0\leq A\leq N)$, which is obtained from \eqref{eq31}.

(N2) $|\nabla \rho|=0$ on $M$, which is obtained by using \eqref{eq33}, Proposition \ref{prop3.2} and the assumption that $T'$ is positive definite.

(N3) $H_T=(\tr T) (\bar{\nabla}\rho)^\bot$ on $M$, which is obtained from \eqref{eq34}.

First, we prove that $\tr T$ is constant.
By using the condition (N1) and the fact that $\sum_{A=0}^N(\Phi^A)^2=1$, we obtain that
\begin{equation}\label{trtc1}
0=\frac{1}{2}L_T(\sum_{A=0}^N(\Phi^A)^2)=\sum_{A=0}^N\Phi^AL_T\Phi^A-\sum_{A,i,j}T_{ij}\Phi^A_i\Phi^A_j=\lambda_2^{L_T}-e^{2\rho}\tr T,
\end{equation}
where we used condition (N1) and \eqref{eq32} in the last equality. On the other hand, condition (N2) implies that $\rho|_M$ is constant, so
we immediately get that $\tr T=\lambda_2^{L_T}/e^{2\rho}$ is a positive constant from \eqref{trtc1}.

In order to prove that $M$ is $T$-minimal in a geodesic sphere $\Sigma_c$ of $\R$, where the geodesic radius  $r_c$ of $\Sigma_c$ is given by
$r_0=\left(\frac{\tr T}{\lambda_2^{L_T}}\right)^{1/2},~r_1=\arcsin r_0,~r_{-1}=\arsinh r_0$,
we need to write down the immersion $\Phi$ explicitly.
With out loss of generality, we can assume that $c=1, 0$ or $-1$.
Recall that  (cf. \cite{MR86,LW12}) for each $g\in B^{N+1}$, we can define a conformal map $\gamma$ on $\mathbb{S}^N(1)$:
\begin{equation}\label{eq_gamma}
\gamma_g(x)=\frac{x+(\mu f+\lambda)g}{\lambda(1+f)}, \quad \forall x\in \mathbb{S}^N(1),
\end{equation}
where $B^{N+1}$ is the open unit ball in $\mathbb{R}^{N+2}$, $x$ is the position vector of $\mathbb{S}^N(1)$, and
\begin{equation}\label{eq_f}
\lambda=(1-|g|^2)^{-1/2},\quad\mu=(\lambda-1)|g|^{-2},\quad f(x)=\l x, g\r.
\end{equation}
When $g=0$, we set $\lambda=1,\mu=0, \gamma_0(x)=x.$

We deal with the three cases ($c=1, 0$ or $-1$) respectively.

\textbf{Case 1. $c=1.$} In this case, the conformal map $\Gamma: \mathbb{S}^N(1)\to \mathbb{S}^N(1)$ in Lemma \ref{lem3.1} is given by
\begin{equation}\label{conf1}
\Gamma=\gamma_g
\end{equation}
for certain $g\in B^{N+1}$, where $x$ is the position vector of $\mathbb{S}^N(1)$ (cf. \cite{ESI00,MR86,LW12}).

We denote $h_1$ the standard metric on $\mathbb{S}^N(1)$, and set $\Gamma^*h_1=e^{2\rho}h_1$, then
by direct computation, we can obtain
\begin{equation}\label{eq_rho1}
e^{2\rho}=\frac{1}{\lambda^2(1+f)^2},\quad
\rho=-\ln \lambda-\ln(1+f), \quad \rho_{A}=-\frac{f_A}{1+f}\ (1\leq A\leq N).
\end{equation}
From the lase equation of \eqref{eq_rho1}, it is obvious that $\rho$ is constant if and only if $f$ is constant.

First, we note that condition (N3) and the assumption that $H_T$ is not identically zero imply that $\rho$ is not constant on $\mathbb{S}^N(1)$, then we have $g\neq 0$. Otherwise, if $g=0$, then $\lambda=1,\mu=0, \gamma_0(x)=x,$ which means that $\gamma_0$ is the identity map, so we have that $\rho\equiv 0$ on $\mathbb{S}^N(1)$, which is a contradiction.
Next, as $\rho$ is not constant on $\mathbb{S}^N(1)$, condition (N2) implies that $M$ lies in  a level set $\{x\in \mathbb{S}^N(1) \mid \rho(x) = b\}$ for some constant $b$. We note that $\rho$ is constant if and only if $f$ is constant, $g\neq 0$, so we get that $M$ lies in a totally umbilical hypersurface
$\{x\in \mathbb{S}^N(1)\mid f(x)=\l x,g\r=a\}$ of $\mathbb{S}^N(1)$ for some constant $a$.
We parameterize $\mathbb{S}^N(1)$ by the geodesic polar coordinates $(r, s^1,\cdots,s^{N-1})$ centered at the north pole $(0,\cdots, 0, 1)$,  where $(s^1,\cdots,s^{N-1})$ are the spherical coordinates on $\mathbb{S}^{N-1}(1)$. Under the geodesic polar coordinates, the position vector can be written as $x=(\tilde{x},\cos r)$ with
$|\tilde{x}|^2=\sin^2 r$. Up to an isometry of $\mathbb{S}^N(1)$, we can assume that $g=(0,g^0)$ and $M$ lies in
$$\Sigma_1=\big\{x\in \mathbb{S}^N(1) \mid x^0=\cos r_1\big\}$$
for some constant $r_1\in (0,\pi)$. Without loss of generality,  we assume that $r_1\in (0, \pi/2]$ in the following.

Next, we prove that $M$ is $T$-minimal in $\Sigma_1$.
Since $g=(0,g^0)$, from the expression of $\rho$ (see \eqref{eq_rho1}) and the definition of $f$ (see \eqref{eq_f}),
we find that $\rho$ only depends on $r$ under the geodesic polar coordinates. As $r=r_1$ is constant on $\Sigma_1$, we know that
when restricted to $\Sigma_1$, $\bar{\nabla}\rho$ lies in the normal bundle of $\Sigma_1$ in $\mathbb{S}^N(1)$, which combines with the
condition (N3) imply that $M$ is $T$-minimal in $\Sigma_1$.

Finally, we determine the geodesic radius of $\Sigma_1$.
The following lemma is the key step to fix the radius for the case $c=1$.
\begin{lem}\label{lem5-1}
	Let $\nu=-\p_r$ and $k$ be the  unit normal vector and the principal curvature of $\Sigma_1$ in $\mathbb{S}^N(1)$ respectively, then we have
	\begin{equation*}
	\left.\bar{\nabla}_\nu\rho\right|_{\Sigma_{1}}=k.
	\end{equation*}
\end{lem}
\begin{proof}
As choose $\nu=-\p_r$ on $\Sigma_1$,  the principal curvature $k=\frac{\cos r_1}{\sin r_1}$. Note that the function $\rho$ only  depends on $r$.  From \eqref{eq_rho1}, we have
	\begin{equation*}
	\bar{\nabla}_\nu\rho=\p_r\ln(1+g^0\cos r)
	=\frac{-g^0\sin r }{(1+g^0\cos r)}.
	\end{equation*}
	
On the other hand, as $\lambda,~\mu,~g^0$ are constant, $x^0$ and $f$ are constant on $M$, we have that $\Phi^0=\frac{x^0+\mu(x^0g^0+\lambda)g^0}{\lambda(1+f)}$ is also constant on $M$ and hence $L_T\Phi^0=0$. We claim that $\Phi^0$ must be $0$ on $M$.
If $\Phi^0\neq 0$, then $\Phi^0$ is the 1st eigenfunction of $L_T$, which contradicts condition (N1). So  $\Phi^0$ must be $0$ on $M$, that is,
	\begin{equation}\label{g01}
		x^0+\mu(x^0g^0+\lambda)g^0=0.
	\end{equation}
When $r_1\in (0, \pi/2)$, note that $g=(0,g^0), |g|^2=(g^0)^2$, $x^0=\cos r_1>0$, we immediately get that $g^0=-x^0=-\cos r_1$ on $M$. Therefore, we have
	\begin{equation*}
	\left.\bar{\nabla}_\nu\rho\right|_{\Sigma_{1}}=\frac{\cos r_1}{\sin r_1}=k.
	\end{equation*}
When $r_1=\pi/2$, we have $x^0=0$, hence from \eqref{g01}, we get that $g^0$ must be $0$, which contradicts $g\neq 0$. Hence, the case $r=\pi/2$ cannot occur.
\end{proof}

Now, we  determine the geodesic radius of $\Sigma_1$.
  We denote the sectional curvature of $\Sigma_c$ by $c'$, then from Gauss equation, we have that $c'=c+k^2$.
  By using Lemma \ref{lem5-1} and condition (N3), we derive that $|H_T|^2=k^2 (\tr T)^2$, as we have proved that $\tr T$ is constant, we get that the right hand side of $\eqref{eq_thm1}$ equals $(1+k^2)\tr T=c' \tr T$. Since that the equality in \eqref{eq_thm1} is attained, we derive that $\lambda_2^{L_T}=c' \tr T$, hence $c'=\lambda_2^{L_T}/\tr T$, then by using
the relation between the geodesic radius $r_c$ and the sectional curvature $c'$ of $\Sigma_c$ (see  \eqref{eq_rad1}), we obtain that the geodesic radius $r_1$ of $\Sigma_1$ is given by $r_1=\arcsin r_0$ with $r_0=\left(\frac{\tr T}{\lambda_2^{L_T}}\right)^{1/2}$.  This completes the proof of the necessary condition for the equality in \eqref{eq_thm1} in the case $c=1$.

\begin{rem}\label{rem3.6}
Once we have obtained that $g^0=-x^0=-\cos r_1$ on $M$,  the radius $r_1$ can also be solved from the expression of $\rho$. However, this method also need to use the observation that $\Phi^0=0$, hence, it is essentially the same as the method presented above.
\end{rem}

\textbf{Case 2. $c=0$.}
There is a conformal map $\pi_0: \mathbb{R}^N \to \mathbb{S}^N(1)\subset \mathbb{R}^{N+1}$ given by ``stereographic projection":
\begin{equation}\label{eq_pi0}
\pi_0(x)=\Big(\frac{2x}{1+|x|^2},\frac{|x|^2-1}{1+|x|^2}\Big) \in \mathbb{S}^N,
\end{equation}
where $x$ is the position vector in $\mathbb{R}^N$.
In this case, the conformal map $\Gamma: \mathbb{R}^N\to \mathbb{S}^N(1)$ in Lemma \ref{lem3.1} is given by (cf. \cite{ESI00})
\begin{equation}\label{conf0}
\Gamma=\gamma_g\circ \pi_0
\end{equation}
for certain $g\in B^{N+1}$.

We denote $h_0$ the standard metric on $\mathbb{R}^N$, and set $\Gamma^*h_1=\pi_0^*\big(\gamma_g^*h_1\big)=e^{2\rho}h_0$. By direct computation, we have
\begin{equation}\label{eq_rho0}
e^{2\rho}= \frac{4}{(1+|x|^2)^2}\cdot\frac{1}{\lambda^2\big(1+f\circ\pi_0(x)\big)^2},
\end{equation}
where $f:\mathbb{S}^N(1)\to \mathbb{R}$ is defined by \eqref{eq_f}.
From \eqref{eq_rho0}, we have that
$\rho$ is constant if and only if $\big(1+f\circ\pi_0(x)\big)(1+|x|^2)=a$, where $a$ is a constant. Suppose $g=(\tilde{g},g^0)\in \mathbb{R}^N\times\mathbb{R}^1$, then we have $1+|x|^2+2\l x,\tilde{g}\r +(|x|^2-1)g^0=a$, which is equivalent to
\begin{equation}\label{eq_sec0}
\left| x+\frac{\tilde{g}}{1+g^0}\right|^2=\frac{(1+g^0)a+|g|^2-1}{(1+g^0)^2},
\end{equation} which means that $M$ lies in an $(N-1)$-dimensional geodesic sphere $\Sigma_0$ of $\mathbb{R}^N$.
We parameterize $\mathbb{R}^N$ by the geodesic polar coordinates $(r, s^1,\cdots,s^{N-1})$ centered at the origin $(0,\cdots,  0)$,  where $(s^1,\cdots,s^{N-1})$ are the spherical coordinates on $\mathbb{S}^{N-1}(1)$. Up to an isometry of $\mathbb{R}^N$, we can assume that
\begin{equation}\label{eq_sig0}
\Sigma_0=\big\{x\in  \mathbb{R}^N\mid |x|=r_0\big\}
\end{equation}
for some $r_0>0$.
Then by repeating the argument above, there exists an element $g=(\tilde{g},g^0)\in B^{N+1}$ such that \eqref{eq_sec0} holds.

Now, we prove that $M$ is $T$-minimal in $\Sigma_0$ and determine the geodesic radius of $\Sigma_0$.
If $g=(0,g^0)$, from the expression of $\rho$ (see \eqref{eq_rho0}) and the definition of $f$ and $\pi_0$ (see \eqref{eq_f} and \eqref{eq_pi0}),
we find that $\rho$ only depends on $r$ under the geodesic polar coordinates, as $r=r_0$ is constant on $\Sigma_0$, we know that
when restricted to $\Sigma_0$, $\bar{\nabla}\rho$ lies in the normal bundle of $\Sigma_0$ in $\mathbb{R}^N$, which combines with the
condition (N3) imply that $M$ is $T$-minimal in $\Sigma_0$.
The following lemma is the key step to fix the radius for the case $c=0$.
\begin{lem}\label{lem5-0}
	Let $\nu=-\p_r$ and $k$ be the unit normal vector and the  principal curvature of $\Sigma_0$ in $\mathbb{R}^N$ respectively. If $g=(0,g^0)$, then we have
	\begin{equation*}
	\left.\bar{\nabla}_\nu\rho\right|_{\Sigma_{0}}=k.
	\end{equation*}
\end{lem}
\begin{proof}
	As we choose $\nu=-\p_r$ on $\Sigma_0$, the principal curvature $k=\frac{1}{r_0}$. Note that the function $\rho$ only  depends on $r$.  By using \eqref{eq_f}, \eqref{eq_pi0}, \eqref{eq_rho0} and the fact that $|x|^2=r^2$, we have
	\begin{equation*}
	\bar{\nabla}_\nu\rho=\p_r\Bigg(\ln(1+r^2)+\ln\Big(1+\frac{r^2-1}{1+r^2}g^0\Big)\Bigg)\\
	=\frac{2r}{1+r^2}+\frac{\frac{4r}{(1+r^2)^2}g^0}{1+\frac{r^2-1}{1+r^2}g^0}.
	\end{equation*}
	
	On the other hand, $f\circ\pi_0(x)=\frac{r_0^2-1}{1+r_0^2}g^0$ is constant  restrict to  $M$, so $$\Phi^0=\frac{\frac{r_0^2-1}{1+r_0^2}+(\mu\, f\circ\pi_0(x)+\lambda)g^0}{\lambda\big(1+f\circ\pi_0(x)\big)}$$ is also constant  on $M$. By an argument similar to that in the proof of Lemma \eqref{lem5-1}, we deduce from $L_T\Phi^0=0$ that $\Phi^0=0$, from which we obtain that $g^0=-\frac{r_0^2-1}{1+r_0^2}$. Hence
		$
	\left.\bar{\nabla}_\nu\rho\right|_{\Sigma_{0}}=1/r_0=k.$
\end{proof}

  We denote the sectional curvature of $\Sigma_0$ by $c'$, then from Gauss equation, we have that $c'=k^2$.
  If $g=(0,g^0)$, then  by using Lemma \ref{lem5-0} and condition (N3), we derive that $|H_T|^2=k^2 (\tr T)^2$, as we have proved that $\tr T$ is constant, we get that the right hand side of $\eqref{eq_thm1}$ equals $k^2\tr T=c' \tr T$. Since that the equality in \eqref{eq_thm1} is attained, we derive that $\lambda_2^{L_T}=c' \tr T$, hence $c'=\frac{\lambda_2^{L_T}}{\tr T}$, then by using
the relation between the geodesic radius $r_c$ and the sectional curvature $c'$ of $\Sigma_c$ (see  \eqref{eq_rad1}), we obtain that the geodesic radius $r_0$ of $\Sigma_0$ is given by $r_0=\left(\frac{\tr T}{\lambda_2^{L_T}}\right)^{1/2}$.

If $g\neq(0,g^0)$, then from \eqref{eq_sec0} and \eqref{eq_sig0}, we know that $M$ lies in an $(N-2)$-sphere $\Sigma_0'$, which is the intersection of two $(N-1)$-spheres. We note that $\Sigma_0'$ can be regarded as a hypersphere in $\mathbb{R}^{N-1}$, where $\mathbb{R}^{N-1}$ is totally geodesic in $\mathbb{R}^N$. Therefore, we reduce the dimensions $N$ and $p$ to $N-1$ and $p-1$ respectively, so we can repeat the  discussions above up to finite times  and finally obtain that  $M$ is a $T$-minimal submanifold in a geodesic sphere $\Sigma_0$ of $\mathbb{R}^N$ with geodesic radius $r_0=\left(\frac{\tr T}{\lambda_2^{L_T}}\right)^{1/2}$.
This completes the proof of the necessary condition for the equality in \eqref{eq_thm1} in the case $c=0$.

\textbf{Case 3. $c=-1$.}
Let $\mathbb{R}^{N+1}_1$ be the Lorentz space equipped with the Lorentz metric $ds^2=dx_1^2+\cdots+dx_N^2-dx_{0}^2$, and denote the inner product by $\l,\r'$, i.e.,
\begin{equation*}
\l x,y\r '=\l \tilde{x},\tilde{y}\r-x^0y^0
\end{equation*}
for $x=(\tilde{x},x^0),y=(\tilde{y},y^0)\in \mathbb{R}^{N+1}_1$. Then
\begin{equation*}
\mathbb{H}^N(-1)=\{x\in \mathbb{R}^{N+1}_1\mid \l x,x\r'=-1, x^0\geq 1\},
\end{equation*}
which is equipped with the induced metric from $\mathbb{R}^{N+1}_1$.

There is a conformal map $\pi: \mathbb{H}^N(-1) \to B^N(1)\subset \mathbb{R}^N$ by ``stereographic projection":
\begin{equation}
\label{eq_pi1}
\pi(x)=\frac{\tilde{x}}{1+x^0}=:w(x) \in B^N(1),\quad
\pi^{-1}(w)=\Big(\frac{2w}{1-|w|^2},\frac{1+|w|^2}{1-|w|^2}\Big)\in \mathbb{R}^{N+1}_1,
\end{equation}
where $x=(\tilde{x},x^0)$ is the positive vector in $\mathbb{H}^{N}(-1)$. In fact, $\Big(B^N(1), \frac{4 |dw|^2 }{(1-|w|^2)^2}\Big)$, which has constant curvature $-1$, is the Poincar\'{e} model of hyperbolic space $\mathbb{H}^N(-1)$. We have the following basic facts.

\begin{lem}\label{lem_hy}
	If $S$ is a hypersphere in $B^N(1)$, then $\pi^{-1}(S)$ is a geodesic sphere of $\mathbb{H}^N(-1)$.
Up to an isometry of $\mathbb{H}^N(-1)$, we can assume that $\pi^{-1}(S)=\{x\in \mathbb{H}^N(-1)\mid x^0=d>1\}$, where $d$ is a constant.
\end{lem}

In this case,
the conformal map $\Gamma: \mathbb{H}^N(-1)\to \mathbb{S}^N(1)$ in Lemma \ref{lem3.1} is given by (cf. \cite{ESI00})
\begin{equation}\label{conf-1}
\Gamma=\gamma_g\circ \pi_0\circ \pi
\end{equation}
for certain $g\in B^{N+1}$.

We denote $h_{-1}$ the standard metric on $\mathbb{H}^N(-1)$, and set $\Gamma^*h_1=\pi^*\Big(\pi_0^*\big(\gamma_g^*h_1\big)\Big)=e^{2\rho}h_{-1}$. By direct computation, we have
\begin{equation}
e^{2\rho}= \frac{(1-|w(x)|^2)^2}{(1+|w(x)|^2)^2}\cdot\frac{1}{\lambda^2\big(1+f(\pi_0(w(x)))\big)^2},
\label{eq_rho-1}
\end{equation}
where $f:\mathbb{S}^N(1)\to \mathbb{R}$, $\pi_0: \mathbb{R}^N\to \mathbb{S}^N(1)$ and $w(x)$ are defined by \eqref{eq_f}, \eqref{eq_pi0}  and \eqref{eq_pi1} respectively.

From \eqref{eq_rho-1}, we know that $\rho$ is constant if and only if $(1+f(\pi_0(w(x))))(1+|w(x)|^2)=a(1-|w(x)|^2)$, where $a$ is a constant.
Suppose $g=(\tilde{g},g^0)\in \mathbb{R}^N\times\mathbb{R}^1$, then we have
\begin{equation}\label{eq_a}
	a=\frac{1+|w|^2+2\l w, \tilde{g}\r}{1-|w|^2}-g^0,
\end{equation}
which implies that $1+a+g^0=\frac{2(1+\l w, \tilde{g}\r)}{1-|w|^2}>0$, where we used $|\tilde{g}|<1,|w|<1$.  Hence, we obtain that
\begin{equation}\label{eq_sec-1}
\left| w(x)+\frac{\tilde{g}}{1+a+g^0}\right|^2=1-\frac{2}{1+a+g^0}+\frac{|\tilde{g}|^2}{(1+a+g^0)^2}.
\end{equation}
Note that $1-\frac{2}{1+a+g^0}+\frac{|\tilde{g}|^2}{(1+a+g^0)^2}<1-\frac{2|\tilde{g}|}{1+a+g^0}+\frac{|\tilde{g}|^2}{(1+a+g^0)^2}=\big(1-\frac{|\tilde{g}|}{1+a+g^0}\big)^2$, thus \eqref{eq_sec-1} implies that $w(x)$ lies in a hypersphere $S$ of $B^N(1)$.

We parameterize $\mathbb{H}^N(-1)$ by the geodesic polar coordinates $(r, s^1,\cdots,s^{N-1})$ centered at the $(0,\cdots, 0, 1)$,  where $(s^1,\cdots,s^{N-1})$ are the spherical coordinates on $\mathbb{S}^{N-1}(1)$. Under the geodesic polar coordinates, the position vector can be written as $x=(\tilde{x},\cosh r)$ with
$|\tilde{x}|^2=\sinh^2 r$.  According to Lemma \ref{lem_hy}, up to an isometry of $\mathbb{H}^N(-1)$, we can assume that $M$ lies in
\begin{equation}\label{eq_sig-1}
\Sigma_{-1}=\pi^{-1}(S)=\{x\in \mathbb{H}^N(-1)\mid x^0=\cosh r_{-1}\}
\end{equation}
for some constant $r_{-1}>0$, so we obtain that  $|\tilde{x}|^2=\sinh^2 r_{-1}$ on $\Sigma_{-1}$, hence from \eqref{eq_pi1} we get that $|w(x)|^2=|\tilde{x}|^2/(1+x^0)^2=\tanh^2\frac{r_{-1}}{2}$ on $\Sigma_{-1}$.
Then by repeating the argument above, there exists an element $g=(\tilde{g},g^0)\in B^{N+1}$ such that \eqref{eq_sec-1} holds.

Now, we prove that $M$ is $T$-minimal in $\Sigma_{-1}$ and determine the geodesic radius of $\Sigma_{-1}$.
If $g=(0,g^0)$, from the expression of $\rho$ (see \eqref{eq_rho-1}) and the definition of $f$, $\pi_0$ and $\pi$ (see \eqref{eq_f}, \eqref{eq_pi0} and  \eqref{eq_pi1}),
we find that $\rho$ only depends on $r$ under the geodesic polar coordinates, as $r=r_{-1}$ is constant on $\Sigma_{-1}$, we know that
when restricted to $\Sigma_{-1}$, $\bar{\nabla}\rho$ lies in the normal bundle of $\Sigma_{-1}$ in $\mathbb{H}^N(-1)$, which combines with the
condition (N3) imply that $M$ is $T$-minimal in $\Sigma_{-1}$.
The following lemma is the key step to fix the radius for the case $c=-1$.

\begin{lem}\label{lem5--1}
	Let $\nu=-\p_r$ and $k$ be the unit normal vector and the principal curvature of $\Sigma_{-1}$ in $\mathbb{H}^N(-1)$ respectively. If $g=(0,g^0)$, then we have
	\begin{equation*}
	\left.\bar{\nabla}_\nu\rho\right|_{\Sigma_{-1}}=k.
	\end{equation*}
\end{lem}
\begin{proof}
	As we choose $\nu=-\p_r$  on $\Sigma_{-1}$,  the principal curvature  $k=\frac{\cosh r_{-1}}{\sinh r_{-1}}.$
Note that the function $\rho$ only  depends on $r$.  By using  \eqref{eq_f}, \eqref{eq_pi0}, \eqref{eq_pi1} and \eqref{eq_rho-1}, we have
	\begin{align*}	 \bar{\nabla}_\nu\rho=&\p_r\Bigg(\ln\big(1+\tanh^2\frac{r}{2}\big)-\ln\big(1-\tanh^2\frac{r}{2}\big)+\ln\Big(1+\frac{\tanh^2\frac{r}{2}-1}{1+\tanh^2\frac{r}{2}}g^0\Big)\Bigg)\\
	=&\p_r\Bigg(\ln(\cosh  r)+\ln\Big(1-\frac{g^0}{\cosh r}\Big)\Bigg)=\tanh r\Big(1+\frac{g^0}{\cosh r-g^0}\Big).
	\end{align*}
	By an argument similar to that in the proof of Lemma \eqref{lem5-1}, we deduce from $L_T\Phi^0=0$ that $\Phi^0=0$, from which we obtain that $g^0=\frac{1}{\cosh r_{-1}}$. Hence
	$$
	\left.\bar{\nabla}_\nu\rho\right|_{\Sigma_{-1}}=\frac{\cosh r_{-1}}{\sinh r_{-1}}=k.
	$$
\end{proof}

  We denote the sectional curvature of $\Sigma_{-1}$ by $c'$, then from Gauss equation, we have that $c'=-1+k^2$.
  If $g=(0,g^0)$, then  by using Lemma \ref{lem5--1} and condition (N3), we derive that $|H_T|^2=k^2 (\tr T)^2$, as we have proved that $\tr T$ is constant, we get that the right hand side of $\eqref{eq_thm1}$ equals $(-1+k^2)\tr T=c' \tr T$. Since that the equality in \eqref{eq_thm1} is attained, we derive that $\lambda_2^{L_T}=c' \tr T$, hence $c'=\frac{\lambda_2^{L_T}}{\tr T}$, then by using
the relation between the geodesic radius $r_c$ and the sectional curvature $c'$ of $\Sigma_c$ (see  \eqref{eq_rad1}), we obtain that the geodesic radius $r_{-1}$ of $\Sigma_{-1}$ is given by  $r_{-1}=\arsinh r_0$ with $r_0=\left(\frac{\tr T}{\lambda_2^{L_T}}\right)^{1/2}$.

If $g\neq(0,g^0)$, then from \eqref{eq_sec-1} and \eqref{eq_sig-1}, we know that $M$ lies in an $(N-2)$-sphere $\Sigma_{-1}'$, which is the intersection of two $(N-1)$-spheres. We note that an $(N-2)$-sphere $\Sigma_{-1}'$ can be regarded as a hypersphere in $\mathbb{H}^{N-1}(-1)$, where $\mathbb{H}^{N-1}(-1)$ is totally geodesic in $\mathbb{H}^{N}(-1)$. Therefore, we reduce the dimensions $N$ and $p$ to $N-1$ and $p-1$ respectively, so we can repeat the  discussions above up to finite times  and finally obtain that $M$ is a $T$-minimal submanifold in a geodesic sphere $\Sigma_{-1}$ of $\mathbb{H}^{N}(-1)$ with geodesic radius $r_{-1}=\arsinh r_0$, where $r_0=\left(\frac{\tr T}{\lambda_2^{L_T}}\right)^{1/2}$.
This completes the proof of the necessary condition for the equality in \eqref{eq_thm1} in the case $c=-1$.
\subsection{A generalization of Theorem \ref{thm1.1}}
In this section,  we generalize Theorem \ref{thm1.1} to Schr{\"o}dinger-type operators. We prove the following result.
\begin{thm}\label{thm3.1}
		Let $M$ be an $n(>2)$-dimensional  closed orientable submanifold in an $(n+p)$-dimensional space form $\R$. Let $T$ be a symmetric, divergence-free $(1,1)$-tensor on $M$. Assume that $T$ is positive definite, and $T'=(\tr T) I-2T$ is semi-positive definite.  For any function  $q \in C^{\infty}(M)$,   consider the operator  $L_{T,q}=L_T+q$, we have the following sharp estimate for the second eigenvalue of $L_{T,q}$:
	\begin{align}\label{eqthm3.14}
	\la_2^{L_{T,q}}\leq\frac{1}{V(M)}\int_M\left(c\tr T+\frac{|H_T|^2}{\tr T}\right)\,dv_M+\bar{q},
	\end{align}
		where $\bar{q}=\frac{1}{V(M)}\int_M q\,dv_M$.	
	The equality in \eqref{eq_thm1} holds  if the following two conditions hold:
	
	(1) $M$ is $T$-minimal in a geodesic sphere $\Sigma_c$ with constant curvature $c'$ of $\R$;
	
	(2) $c'\tr T+ q=\lambda_2^{L_{T,q}}$ is constant.
	
	Moreover, if $T'$ is positive definite and $H_T$ is not identically zero, then the equality  in \eqref{eq_thm1} holds if and only if  the conditions (1) and (2) hold.
\end{thm}
\begin{proof}
Under the assumptions in Theorem \ref{thm1.1},   there exists a regular conformal map $\Gamma: \R\to \S\subset\mathbb{R}^{n+p+1}$ such that the components of the immersion $\Phi=\Gamma\circ x=(\Phi^1, \cdots, \Phi^{n+p+1})$ are $L^2$-orthogonal to the first eigenfunction of $L_{T,q}$ (cf. \cite{ESI92a}], then by using similar argument to that in Sections 3.1, it is not hard to prove that
	\begin{equation*}
		\la_2^{L_{T,q}}=\frac{1}{V(M)}\int_Me^{2\rho}\tr T\,dv_M+\bar{q}
		\leq\int_M \left(c\tr T+\frac{|H_T|^2}{\tr T}\right)\,dv_M+\bar{q}.
	\end{equation*}
	
If the conditions (1) and (2) hold, then  by analogous argument to that in Section 3.2,  we obtain that the equality in \eqref{eqthm3.14} is attained.
	
Conversely, if $T'$ is positive definite, $H_T$ is not identically zero and the equality  in \eqref{eq_thm1} holds,  then by modifying the  arguments in Section 3.3 slightly, we can obtain that the conditions (1) and (2) hold.  We briefly explain the difference.
First, the condition (N1) will be replaced by

(N1') $L_{T,q}\Phi^A=\lambda_2^{L_{T,q}}\Phi^A$ for all $0\leq A\leq N$.
 The other two conditions (N2) and (N3) are the same. Then by using the condition (N1') and the fact that $\sum_{A=0}^N(\Phi^A)^2=1$, we obtain that
\begin{equation}\label{trtc2}
\begin{aligned}
q&=\sum_{A=0}^N\Big(\frac{1}{2}L_{T}+q\Big)(\Phi^A)^2=\sum_{A=0}^N\Phi^AL_{T,q}\Phi^A-\sum_{A,i,j}T_{ij}\Phi^A_i\Phi^A_j\\
&=\lambda_2^{L_{T,q}}-e^{2\rho}\tr T,
\end{aligned}
\end{equation}
where we used condition (N1') and \eqref{eq32} in the last equality.
From \eqref{trtc2}, we immediately get that $e^{2\rho}\tr T+q=\lambda_2^{L_{T,q}}$ is a positive constant. In order to prove the condition (2), it suffices to show that $e^{2\rho}=c'$. We explain how to show this in the case $c=1$. From the expression of $\rho$ and conditions (N2) and (N3), we can first also show  that $M$ is $T$-minimal in a geodesic sphere $\Sigma_{1}$, which is the condition (1).  By adjusting the proof of Lemma \ref{lem5-1}, we have that $\Phi^0$ is constant on $M$ and $L_{T,q}\Phi^0=q\Phi^0=\lambda_2^{L_{T,q}}\Phi^0$. If $\Phi_0\neq 0$, then  $q=\lambda_2^{L_{T,q}}$ is constant. But we know that $\lambda_1^{L_{T,q}}=q$ if $q$ is constant and $\lambda_1$ is simple, which is a contradiction. Hence, $\Phi_0$ much be $0$, then we get that $g^0=-x^0=-\cos r_1$ on $M$, from this we derive that $e^{2\rho}=1/\sin^2 r_1=c'$.
The case $c=0$ and the case $c=-1$
can be proved by modifying the arguments in Section 3.3 in a similar way.
\end{proof}

\begin{rem}
If we take $L_{T}=-\Delta$, then we can obtain Theorem 2.1 in \cite{ESI00} by applying Theorem \ref{thm3.1}.
When $p=1$, $L_T=L_r$, which is the linearized operators for the first variation of the $(r+1)$-mean curvature for hypersurfaces in a space form,
the Jacobi operator for the corresponding  variational problem is a Schr{\"o}dinger-type operator associated with $L_r$. In \cite{LW12}, the estimate of the second eigenvalue of the Jacobi
operator $J_s$ for hypersurfaces with constant scalar curvature in a sphere was applied to give a new proof of the stability result of \cite{AdCC93}.
 We expect that our Theorem \ref{thm3.1} can be applied to prove some stability result for some variational problems in higher codimension cases in future.
\end{rem}

\section{Application of Theorem \ref{thm1.1} to the $L_r$ operator}

In this section,  we prove Theorem \ref{cor3.15} by applying Theorem \ref{thm1.1} to the $L_r$ operator for the case $p=1$ or the case $p>1$ and $r$ is even. We also give some examples which satisfy the assumptions in Theorem \ref{cor3.15} and attain the equality in \eqref{eq3.47}. These examples show that our estimate is really sharp.

\subsection{Proof of Theorem \ref{cor3.15}}
First of all, when $n=2$, as we assume that $r\leq n-2$,  $r$ must be $0$, then $L_0=-\Delta$. Although we assume that $n>2$ in Theorem \ref{thm1.1},
the conclusion is still true for $n=2$ and $L_0=-\Delta$. This can be seen from Remark \ref{remn2} and the proof of Theorem \ref{thm1.1}.
Actually, this case corresponds to the ``Reilly in equality'' which was proved in \cite{Rei77} and \cite{ESI92}.

Now, we assume that $n>2$, we only need to check that under the assumption of Theorem \ref{cor3.15}, the assumptions  in \ref{thm1.1} are satisfied.
If $p>1$ and $r$ is even, or $p=1$, the tensor $T_r$ is a symmetric and divergence-free $(1,1)$-tensor (cf. Lemma 2.1 in \cite{Gro02} or Lemma 3.1 in \cite{CL07}).
	From Lemma \ref{lem2.1}, we have the following relations:
	\begin{align}
	&\tr (T_r)=(n-r) S_r,\label{eq3.45}\\
	&T'_r=\tr (T_r) I - 2 T_{r}=(n-r-2) S_r+(\sum_{k,\alpha}T_{r-1\, kj}^\alpha h_{ki}^\alpha).\nonumber
	\end{align}
	So the assumption that ``$T_{r}$ is positive definite and  $(\sum_{k,\alpha}T_{r-1\, kj}^\alpha h_{ki}^\alpha)$ is semi-positive definite'' implies that $S_r>0$ and $T'_{r}$ is semi-positive definite. Moreover, $T'_r$ is positive definite if $r<n-2$. Hence, $T_r$ satisfies the conditions in Theorem \ref{thm1.1}, then by applying Theorem \ref{thm1.1}, we complete the proof by using \eqref{eq2.23}, \eqref{eq3.45} and the relation $H_{T_r}=(r+1)\mathbf{S}_{r+1}$.

\subsection{Some Examples}
We give some examples in higher codimension case ($p>1$). For simplicity,  we only consider the $L_2$ operator ($r=2$). We always assume that $n\geq 4$ in the following examples since $2=r\leq n-2$.

\begin{ex}[Torus in Euclidean space or the hyperbolic space]\label{ex-torus}
	Assume that $c\leq 0, n\geq 4$ and $1\leq m\leq n/2$. Given $0<a<1$,  let
	$$x: M=\mathbb{T}(m,a)=\mathbb{S}^m(a)\times \mathbb{S}^{n-m}(\sqrt{1-a^2})\to \mathbb{S}^{n+1}(1)\subset \mathbb{S}^{n+p-1}(1)\subset  \R$$
	be a torus immersed in $\R$, where $\mathbb{S}^m(a)$ and  $\mathbb{S}^{n-m}(\sqrt{1-a^2})$ denote a sphere with radius $a$ and $\sqrt{1-a^2}$ respectively.
	
\textbf{Claim 1. For the tensor $T_2$ and the operator $L_2$, all the assumptions in Theorem \ref{cor3.15} are satisfied.}
	We denote the position vector of $\mathbb{T}(m,a)$ in  $\mathbb{S}^{n+1}(1)$ by $x=(x_1,x_2)\in \mathbb{S}^m(a)\times \mathbb{S}^{n-m}(\sqrt{1-a^2})$, then the unit normal vector at this point $x$ is given by $e_{n+1}=(\frac{\sqrt{1-a^2}}{a} x_1,-\frac{a}{\sqrt{1-a^2}}x_2)$. The principal curvatures of $x$ are given by
\begin{equation}
k_1=\cdots=k_m=-\frac{\sqrt{1-a^2}}{a},~k_{m+1}=\cdots=k_n=\frac{a}{\sqrt{1-a^2}}.
\end{equation}
The Ricci curvature of $\mathbb{T}(m,a)$ is non-negative and is given by
\begin{equation*}
R_{ii}=\frac{m-1}{a^2} (1\leq i\leq m), \, R_{jj}=\frac{n-m-1}{1-a^2} (m+1\leq j\leq n),
\end{equation*}
so $\Ric\geq (n-1)c I$ as $c\leq 0$, then by using \eqref{eq3.50}, we obtain that $(\sum_{k,\alpha}T_{r-1\, kj}^\alpha h_{ki}^\alpha)$ is semi-positive definite
From \eqref{eq3.50}, \eqref{eq2.26} and Gauss equation \eqref{gauss3}, we know that
\begin{align*}
	T_2=&[S_2+(n-1)c]I-\Ric
	=\Big[\frac{R-n(n-1)c}{2}+(n-1)c\Big]I-\Ric\\
	=&\frac{1}{2}\Big[\frac{m(m-1)}{a^2}+\frac{(n-m)(n-m-1)}{1-a^2}-(n-1)(n-2)c\Big]I-\Ric>0.
\end{align*}

\textbf{Claim 2. The inequality in \eqref{eq3.47} holds .}

The tensor $T_2$ can be diagonalized as $\mbox{diag}(t,\cdots, t, s, \cdots, s)$, where the first $m$ entires are $t$ and the rest are $s$, and
\begin{align*}
t=&\frac{1}{2}\Big[\frac{(m-2)(m-1)}{a^2}+\frac{(n-m)(n-m-1)}{1-a^2}-(n-1)(n-2)c\Big],\\
s=&\frac{1}{2}\Big[\frac{m(m-1)}{a^2}+\frac{(n-m-2)(n-m-1)}{1-a^2}-(n-1)(n-2)c\Big].
\end{align*}
Using the first (non-zero) eigenvalue of Laplacian on a sphere, we have (cf. Example 3.2 in \cite{LW12})
$$\lambda_2^{L_2}=\min\{mt/a^2, (n-m)s/(1-a^2)\}.$$
Without loss of generality, we assume $mt/a^2\leq (n-m)s/(1-a^2)$, which is equivalent to
\begin{equation}\label{eq46}
\frac{(n-m)s}{mt}\geq\frac{1-a^2}{a^2}.
\end{equation}
On the other hand, if we denote $S_3'$ the $3$-mean curvature of $\mathbb{T}(m,a)$ in $ \mathbb{S}^{n+1}(1)$, then we have
\begin{align}
	&\frac{(n-2)\binom{n}{2}}{V(M)}\int_M\frac{cH_2^2+|\mathbf{H}_{3}|^2}{H_2}\,dv_M\nonumber\\
	=&\tr T_2+\frac{(3S_3')^2}{\tr T_2}
	\geq\tr T_2=mt+(n-m)s\label{eq47}\\
	=&mt\Big(1+\frac{(n-m)s}{mt}\Big)\geq mt\Big(1+\frac{1-a^2}{a^2}\Big)=\lambda_2^{L_2},\label{eq48}
\end{align}
where we used $\eqref{eq46}$ in the last inequality.

\textbf{Claim 3. $\mathbb{S}^m(\sqrt{1/2})\times \mathbb{S}^{m}(\sqrt{1/2})\to \mathbb{R}^N$ or $\mathbb{H}^N(-1)$ attains the  equality in \eqref{eq3.47}.}

From \eqref{eq47} and \eqref{eq48}, we know that the equality in \eqref{eq3.47}   holds if and only if
$S_3'=0$ and $\frac{a^2}{1-a^2}=\frac{mt}{(n-m)s}$.

The principal curvatures of $\mathbb{T}(m,a)$ in $\mathbb{S}^{n+1}(1)$ are $\lambda=-\frac{\sqrt{1-a^2}}{a}$ with multiplicity $m$ and $\mu=\frac{a}{\sqrt{1-a^2}}$ with multiplicity $n-m$, so we get $\lambda\mu=-1, S_3'=\binom{m}{3}\lambda^3+\binom{m}{2}\binom{n-m}{1}\lambda^2\mu+\binom{m}{1}\binom{n-m}{2}\lambda\mu^2+\binom{n-m}{3}\mu^3$.

When $n=2m\geq 4$, It is easy to check that $t=s$,  $\mu^2=\frac{a^2}{1-a^2}=\frac{mt}{(n-m)s}=1, S_3'=0$.

\end{ex}

\begin{ex}[Einstein manifolds]\label{ex-Ein}
	Let $M^n$ be an Einstein manifold with scalar curvature $R>n(n-1)c$ immersed in $\R$, then the operator $L_2$ on $M$ satisfies the assumptions in Theorem \ref{cor3.15}.
	In fact, for an Einstein manifold, $\Ric=\frac{R}{n}I$. It is  not hard to verify that $T_2=\Big[\frac{R-n(n-1)c}{2}+(n-1)c\Big]I-\Ric>0$ if and only if $Ric>(n-1)c I$, i.e., $R>n(n-1)c$. By using \eqref{eq3.50}, we obtain that $(\sum_{k,\alpha}T_{r-1\, kj}^\alpha h_{ki}^\alpha)$ is positive definite
	From
	In addition, if we assume that $M$ is minimal in some sphere $\mathbb{S}^{n+l}(a)\subset\R$ ($a$ denotes the radius of the sphere) and has constant Ricci curvature, then $M$ is also $2$-minimal (i.e., the $3$-rd mean curvature vector $\mathbf{S}_{3}'$ of $M$ in $\mathbb{S}^{n+l}(a)$ vanishes) in the sphere (cf. Example 5.6 in  \cite{CL07}), and $T_2=K I, L_2=-K\Delta$ for some constant $K>0$. Hence,  by combining with the Takahashi theorem, we have
	\begin{equation*}
	\frac{(n-2)\binom{n}{2}}{V(M)}\int_McH_2+\frac{|\mathbf{H}_{3}|^2}{H_2}\,dv_M
	=\frac{\tr T_2}{a^2}+\frac{(3|\mathbf{S}_{3}'|)^2}{\tr T_2}
	=\frac{\tr T_2}{a^2}=\frac{nK}{a^2} \geq K\lambda_2^{\Delta}
	=\lambda_2^{L_2}.
	\end{equation*}
	In particular, when $\lambda_2^{L_2}=n/a^2$, the equality in \eqref{eq3.47} is attained.
	
	(1) \textbf{Spheres with radius $a$}.
	We consider an immersion $x: M=\mathbb{S}^n(a)\to \mathbb{S}^{n+p-1}(a)\subset \R$ (we assume that $a<1$ when $c=1$).
	Note that $M$ is totally geodesic in $\mathbb{S}^{n+p-1}(a)$, so it is automatically $2$-minimal. We also have
  \begin{equation*}
	k^2=\frac{1}{a^2}-c,\quad H_2=k^2,\quad |\mathbf{H}_3|=k^3.  \quad L_2=-\frac{(n-2)(n-1)}{2}k^2\Delta,~\lambda_2^{\Delta}=n/a^2,
	\end{equation*}
	where $k>0$ is the principal curvature of $\mathbb{S}^{n+p-1}(a)$ in $\R$, hence we get
\begin{equation*}
\lambda_2^{L_2}=\frac{n(n-2)(n-1)}{2a^2}k^2= \frac{(n-2)\binom{n}{2}}{V(M)}\int_M\frac{cH_2^2+|\mathbf{H}_{3}|^2}{H_2}\,dv_M.
\end{equation*}
In fact, for spheres, the equality in \eqref{eq3.47} holds for each $r$.
	
	(2) \textbf{Projective spaces.}
	Let $\mathbb{F}$ denote the field $\mathbb{R}$ of real numbers, the field $\mathbb{C}$ of  complex numbers of the filed $\mathbb{Q}$ of quaternions, and $M=P^m(\mathbb{F})$ be the projective space over $\mathbb{F}$, then the real dimension of $M$ equals $n=m\cdot d_{\mathbb{F}}$, where
	\begin{numcases}{d_{\mathbb{F}}=}
	1, & $\mathbb{F}=\mathbb{R}$,\nonumber\\
		2, & $\mathbb{F}=\mathbb{C}$,\nonumber\\
			4, & $\mathbb{F}=\mathbb{Q}$.\nonumber
	\end{numcases}

	Let $\phi_1: P^m(\mathbb{F})\to \mathbb{S}^{N-1}(\sqrt{\frac{m}{2(m+1)}}) $ be the first standard minimal immersion into a unit sphere (see \cite{Ros83} or Chapter 4.6 in \cite{Che84} for the details), and $i: \mathbb{S}^{N-1}(\sqrt{\frac{m}{2(m+1)}})\to \mathbb{R}^{N}$ be the inclusion map, where $N=\frac{m(m+1)d}{2}+m$. We consider the immersion $x=i\circ\phi_1: P^m(\mathbb{F})\to \mathbb{R}^{N}$, then the curvature of $M$ and the 1st nonzero eigenvalue of Laplacian on $M$ are listed as follows (cf. \cite{Che84,Ros83}).
	
	\begin{tabular}{|c|c|c|c|c|}
	\hline
	$M$ & dim $M$ & sectional curvature  & Ricci curvature & $\lambda_2^{\Delta}$(1st nonzero eigenvalue of $\Delta $) \\
	\hline
	$P^m(\mathbb{R})$ & $m$ & $1$  & $m-1$ &   $2(m+1)$\\
	\hline
	$P^m(\mathbb{C})$ & $2m$ & $[1,4]$ & $2(m+1)$ & $4(m+1)$\\
	\hline
	$P^m(\mathbb{Q})$ & $4m$ & $[1,4]$ & $4(m+2)$ &  $8(m+1)$ \\
	\hline
\end{tabular}
	Since $P^m(\mathbb{F})$ has positive constant Ricci curvature, is minimally immersed in a unit sphere and satisfies that $\lambda_2^\Delta=\text{dim } M\Big/(\sqrt{\frac{m}{2(m+1)}})^2$, we obtain that the equality in \eqref{eq3.47} is attained.
\end{ex}

\section{Application to the operator $L$ defined by \eqref{L1}}

In this section, we prove Theorem \ref{thm1.2} by applying Theorem \ref{thm1.1} to the tensor $T$ defined by \eqref{T1} and the operator defined by \eqref{L1}. Under the assumption $H_2>0$, we have that $n^2H^2>S>0$, so we have that the mean curvature vector $\mathbf{H}$ is nowhere zero. We choose $e_{n+1}=\frac{\mathbf{H}}{H}$ as in Section 1.

First, we prove the following algebraic lemma which will be used in the proof for the case $n=4$.
\begin{lem}\label{lem_alg}
	Given two positive numbers $a>0,b>0$ which satisfy $9a^2>24b$, if we
	consider a function $f: \mathbb{R}^4\to \mathbb{R}$ defined by
	$$f(x)=3x_1+x_2+x_3+x_4,~\forall x=(x_1,x_2,x_3,x_4)\in\mathbb{R}^4,$$
	then
	$f(x)$ is always positive on the set $$K_{a,b}=\{x=(x_1,x_2,x_3,x_4)\in \mathbb{R}^4| \sum_{1\leq i\leq 4}x_i=a, \sum_{1\leq i<j\leq 4}x_ix_j=b\}.$$
\end{lem}
\begin{proof}
	Note that if $x\in K_{a,b}$, then $|x|^2=a^2-2b>0$, so $K_{a,b}$ is bounded. Obviously $K_{a,b}$ is a closed set, hence $K_{a,b}$ is a compact set.
	We will prove that the minimum of $f$ in $K_{a,b}$ is positive. By using the method of Lagrange multipliers,  we consider the function
	$$F(x,\lambda,\mu)=f(x)-\lambda(\sum_{i=1}^4x_i-a)-\mu(\sum_{1\leq i<j\leq 4}^nx_ix_j-b).$$
	The necessary condition for an extremum of $F(x,\lambda,\mu)$ is given by
	\begin{equation*}\left\{\begin{aligned}
	\frac{\p F}{\p x_1}=&3-\lambda-\mu(x_2+x_3+x_4)=0,\\
	\frac{\p F}{\p x_2}=&1-\lambda-\mu(x_1+x_3+x_4)=0,\\
	\frac{\p F}{\p x_3}=&1-\lambda-\mu(x_1+x_2+x_4)=0,\\
	\frac{\p F}{\p x_4}=&1-\lambda-\mu(x_1+x_2+x_3)=0,
	\end{aligned}\right.
	\end{equation*}
which implies that
\begin{equation*}
2+\mu(x_1-x_2)=\mu(x_2-x_3)=\mu(x_2-x_4)=\mu(x_3-x_4)=0,
\end{equation*}
so we get $\mu\neq 0$ and $x_2=x_3=x_4$.
Let $x_1=s, x_2=x_3=x_4=t$, as $x\in K_{a,b}$, we have $s+3t=a, 3st+3t^2=b$, so we get  $6t^2-3at+b=0$.
We can solve out
\begin{equation*}
t=\frac{3a\pm \sqrt{9a^2-24b}}{12}.
\end{equation*}
Note that $s=a-3t$, we get
\begin{equation*}
f(x)=3(s+t)=\frac{3a\mp\sqrt{9a^2-24b}}{2}>0.
\end{equation*}
 Hence, we derive that the minimum of  $f$ in $K_{a,b}$ equals $\frac{3a-\sqrt{9a^2-24b}}{2}$ which is positive.
This completes the proof.
\end{proof}

\noindent \textbf{Proof of Theorem \ref{thm1.2}:}
	From the definitions of  $T$ in \eqref{T1} and  $L$ in \eqref{L1}, we have $T_{ij}=(nH\delta_{ij}-h_{ij}^{n+1})$, then
	\begin{equation}\label{eq5.6}
		\tr T=n(n-1)H,
	\end{equation}
	\begin{align}
		 H_T&=\sum_{i,j,\alpha}(nH\delta_{ij}-h_{ij}^{n+1})h_{ij}^{\alpha}e_{\alpha}=(n^2H^2-\sum_{i,j}(h_{ij}^{n+1})^2)e_{n+1}-\sum_{\alpha\geq n+2}\Big(\sum_{i,j}h_{ij}^{n+1}h_{ij}^{\alpha}\Big)e_{\alpha} \nonumber\\
		&=\Big(n(n-1)H_2+\sum_{\alpha\geq n+2}(h_{ij}^{\alpha})^2\Big)e_{n+1}-\sum_{\alpha\geq n+2}\Big(\sum_{i,j}h_{ij}^{n+1}h_{ij}^{\alpha}\Big)e_{\alpha}.\label{eq5.7}
	\end{align}
	In order to apply Theorem \ref{thm1.1}, we verify that $T$ and $T'=(\tr T)I-2T$ are both positive definite.
	
	We choose orthonormal frame $\{e_1,\cdots, e_n\}$
	 such that $(h_{ij}^{n+1})=\mbox{diag}(k_1,\cdots, k_n)$ is diagonalized. $H_2>0$ implies that
	 \begin{equation*}
0<  n(n-1)H_2=(nH)^2-S\leq (nH)^2-\sum_{i}(k_i)^2
\leq  (nH)^2-(k_i)^2,
	 \end{equation*}
	 So $nH>|k_i|$ for each $i$, which implies that $T$ is positive definite.
On the other hand, the principal curvatures of $T'$ are given by
$T'_{ii}=n(n-1)H-2(nH-k_i)=n(n-3)H+2k_i,~i=1,\cdots,n$.
Assume that $k_1\leq\cdots\leq k_n$ without loss of generality, then it is sufficient to show that $n(n-3)H+2k_1>0$.

When $n\geq 5$, since $nH>|k_1|$, we have
$n(n-3)H+2k_1\geq 2(nH-|k_1|)>0.$

When $n=4$, we need to prove that $4H+2k_1>0$. Set $x_i=k_i (i=1,2,3,4)$ in Lemma \ref{lem_alg}, then by using the Newton-Maclaurin inequality, $9a^2-24b=144(H^2-H_2)>0$. Hence by applying Lemma \ref{lem_alg}, we get  $4H+2k_1=3k_1+k_2+k_3+k_4>0$

Now we can apply Theorem \ref{thm1.1}.
By substituting \eqref{eq5.6} and \eqref{eq5.7} into \eqref{eq_thm1}, we obtain  the inequality \eqref{eqthm}.

If the equality in \eqref{eqthm} holds, we have the following conclusions: $(1)$ $\tr T=n(n-1)H$ is constant; $(2)$ $M$ is $T$-minimal in a geodesic sphere $\Sigma_c$ of $\R$, where the geodesic radius $r_c$ of $\Sigma_c$ is given by \begin{equation*}
r_0=\left(\frac{\tr T}{\lambda_2^{L_T}}\right)^{1/2},\quad r_1=\arcsin r_0,\quad r_{-1}=\arsinh r_0.
\end{equation*}

In the following, we prove that $M$ is not only $T$-minimal in $\Sigma_c$, but also minimal.
Let $\mathbf{H}'$ be the mean curvature vector of $M$ immersed in $\Sigma_c$, denote $H'=|\mathbf{H}'|$. We will show that $H'=0$.
Suppose that $H'\neq 0$, as $e_{n+1}$ is parallel with $\mathbf{H}$, it is obvious that $e_{n+1}\in\text{Span}\{\nu,\mathbf{H}'\}$, where $\nu$ is the unit normal of $\Sigma_c$ in $\R$.  We can choose the normal frame $\{\tilde{e}_{n+1},\cdots,\tilde{e}_{n+p}\}$ on $M$ such that $\tilde{e}_{n+1}=\nu, \tilde{e}_{n+2}=\frac{\mathbf{H}'}{H'}, \tilde{e}_\beta=e_{\beta} (\beta\geq n+3)$.
Then $\tilde{h}_{ij}^{n+1}=\delta_{ij}k$
and $H^2=H'^2+k^2$, where $k$ is the principal curvature of $\Sigma_c$ in $\R$.

We assume that
\begin{equation*}
\begin{pmatrix}
e_{n+1}\\
e_{n+2}
\end{pmatrix} =
\begin{pmatrix}
\cos \theta& \sin\theta  \\
-\sin\theta & \cos \theta
\end{pmatrix} \begin{pmatrix}
\tilde{e}_{n+1}\\
\tilde{e}_{n+2}
\end{pmatrix}.
\end{equation*}
Then we have
\begin{equation}\label{tminsec5}\begin{aligned}
\begin{pmatrix}
h_{ij}^{n+1}\\
h_{ij}^{n+2}
\end{pmatrix} &=
\begin{pmatrix}
\cos \theta& \sin\theta  \\
-\sin\theta & \cos \theta
\end{pmatrix} \begin{pmatrix}
\tilde{h}_{ij}^{n+1}\\
\tilde{h}_{ij}^{n+2}
\end{pmatrix},\\
\tilde{H}^{n+1}&=k=H\cos\theta, \quad
\tilde{H}^{n+2}=H'=H\sin\theta,
\end{aligned}
\end{equation}
where $\tilde{h}_{ij}^{n+1}$ and $\tilde{h}_{ij}^{n+2}$ denote the components of the second fundamental form in the directions of $\tilde{e}_{n+1}$ and $\tilde{e}_{n+2}$ respectively.
By using \eqref{tminsec5} and the $T$-minimality of $M$ in $\Sigma_c$, we get
\begin{equation}\begin{aligned}
0&=\sum_{i,j}T_{ij}\tilde{h}_{ij}^{n+2}=nH\sum_{i}\tilde{h}_{ii}^{n+2}-\sum_{i,j}h_{ij}^{n+1}\tilde{h}_{ij}^{n+2}\\
&=n^2HH'-\sum_{i,j}\cos\theta\tilde{h}_{ij}^{n+1}\tilde{h}_{ij}^{n+2}-\sum_{i,j}\sin\theta\tilde{h}_{ij}^{n+2}\tilde{h}_{ij}^{n+2}\\
&=n^2H^2\sin\theta-nkH'\cos\theta-\sin\theta\sum_{i,j}(\tilde{h}_{ij}^{n+2})^2\\
&=n^2H^2\sin\theta-nk^2\sin\theta-\sin\theta\sum_{i,j}(\tilde{h}_{ij}^{n+2})^2\\
&=n^2H^2\sin\theta-\sin\theta\sum_{i,j}(\tilde{h}_{ij}^{n+1})^2-\sin\theta\sum_{i,j}(\tilde{h}_{ij}^{n+2})^2.\label{eq_45}
\end{aligned}\end{equation}
Since $H>0,~H'>0$, then from \eqref{tminsec5} we get $\sin\theta> 0$, hence \eqref{eq_45} implies that
$0\geq \sin\theta (n^2H^2-S)= n(n-1)H_2\sin\theta $, so we get $H_2\leq 0$, which is a contradiction.

Hence $H'=0$, i.e., $M$ is minimal in $\Sigma_c$. Furthermore, $H=k, L=-(n-1)k\Delta, \lambda_2^{L}=(n-1)k\lambda_2^\Delta, \tr T=n(n-1)H=n(n-1)k$, so from Theorem \ref{thm1.1}, it follows that the geodesic radius of $\Sigma_c$ is $r_c$ is given by
$r_0=\left(\frac{n}{\lambda_2^\Delta}\right)^{1/2}$, $r_1=\arcsin r_0$ and $r_{-1}=\arsinh r_0$.

Conversely, if $M$ is minimal in a geodesic sphere $\Sigma_c$ of radius $r_c$ given as mentioned above, we can  obtain the equality in \eqref{eqthm} by using analogous arguments to that in Section 3.2.
\qed
\appendix
\section{Relations between the $k$-th Gauss-Bonnet curvature and the $r$-th mean curvature}

Recall that under the orthonormal frame, for $1\leq k\leq n/2$, the $k$-th Lovelock curvature $E^{(k)}_{ij}$, the $k$-th Gauss-Bonnet curvature $L_k$ and the $P_{(k)}$ curvature corresponding to $L_k$ are defined by (cf. \cite{Lov71, GWW14})
\begin{align*}
E^{(k)}_{ij}&:=-\sum\frac{1}{2^{k+1}}\delta^{i_1\ldots i_{2k}i}_{j_1\ldots j_{2k}j}R_{i_1i_2j_1j_2}\cdots R_{i_{2k-1}i_{2k}j_{2k-1}j_{2k}} \quad \mbox{($k<n/2$ is required)},\\
L_k&:=\sum\frac{1}{2^k}\delta^{i_1\ldots i_{2k}}_{j_1\ldots j_{2k}}R_{i_1i_2j_1j_2}\cdots R_{i_{2k-1}i_{2k}j_{2k-1}j_{2k}},\\
P_{(k)}^{stlm} &:=\sum\frac{1}{2^k}\delta^{i_1\ldots i_{2k-2}st}_{j_1\ldots j_{2k-2}lm}R_{i_1i_2j_1j_2}\cdots R_{i_{2k-3}i_{2k-2}j_{2k-3}j_{2k-2}}.
\end{align*}
We point out that $L_k$ denotes the $k$-th Gauss-Bonnet curvature in this appendix, and it cannot be confused with the operator $L_r$ defined in Section 2.

These curvature tensors are important in geometry and physics, and have been widely studied by by many mathematicians and physicists. For example, $E^{(k)}_{ij}$, introduced by Lovelock in \cite{Lov71},  is a generalization of the Einstein tensor $E^{(1)}_{ij}=R_{ij}-\frac{R}{2}\delta_{ij}$; $L_1=R$ is the scalar curvature, and $L_{n/2}$ is well-known as the Euler density which appears in the famous Gauss-Bonnet-Chern theorem; $P_{(k)}$ can be used to define the Gauss-Bonnet-Chern mass,  which generalizes the ADM mass. We refer the readers to \cite{GWW14,GWW14a} for more details.

We note that the generalized Kronecker delta has following the properties:
\begin{align}
&\delta^{i_1\ldots i_t \ldots  i_s \ldots i_l}_{j_1\ldots  j_t \ldots  j_s \ldots j_l}=-\delta^{i_1 \ldots i_s \ldots  i_t \ldots i_l}_{j_1\ldots  j_t \ldots  j_s \ldots j_l},\quad \mbox{for $1\leq t<s\leq l\leq n$},\label{eq-kron-prop1}\\
\sum_{i_{t+1},\ldots, i_l}&\delta^{i_1\ldots i_t i_{t+1}\ldots i_l}_{j_1\ldots j_t i_{t+1}\ldots i_l}=\frac{(n-t)!}{(n-l)!}\delta^{i_1\ldots i_t }_{j_1\ldots j_t},\quad  \mbox{for $1\leq t\leq l\leq n$},\label{eq-kron-prop2}\\
&\delta^{i_1\ldots  i_l}_{j_1\ldots j_l}=\delta^{i_l}_{j_l}\delta^{i_1\ldots  i_{l-1}}_{j_1\ldots j_{l-1}}-\sum_{t=1}^{l-1}\delta^{i_l}_{j_t}\delta^{i_1\ldots i_{t-1} i_t i_{t+1}\ldots  i_{l-1}}_{j_1\ldots j_{t-1} j_l j_{t+1}\ldots j_{l-1}},
\quad  \mbox{for $1\leq l\leq n$},\label{eq-kron-prop3}
\end{align}
It is not hard to obtain that
\begin{lem}\label{lem2.2}
	Set $E^{(0)}_{ij}=-\delta_{ij}$. For $1\leq k\leq n/2$, we have
	\begin{align*}
	P_{(k)}^{ijml}&=-P_{(k)}^{jiml}=-P_{(k)}^{ijlm}=P_{(k)}^{mlij},\\
	\sum_{s}P_{(k)}^{sisj}&=-(n-2k+1)E^{(k-1)}_{ij},\\
	E^{(k)}_{ij}&=-(\frac{1}{2}L_k\delta_{ij}-k\sum_{s,t,l}P_{(k)}^{stli}R_{stlj}),\\
	L_k&=\sum_{s,t,l,m}P_{(k)}^{stlm}R_{stlm},\\
	\tr E^{(k)}&=-\frac{n-2k}{2}L_k \quad \mbox{($k<n/2$ is required)}.
	\end{align*}
\end{lem}

Now we show some relations between $P_{(k)}, L_k$ and $T_{r-1}$.

\begin{lem}\label{lem2.3}
	For a submanifold $M^n$ in $\R$ and $r=2k$ is even, we have
	\begin{align}
	\sum_{m,\alpha}T_{r-1\,mj}^{\alpha}h_{mi}^\alpha=
	&\frac{1}{(2k-1)!}\Bigg\{\frac{1}{2^k}\frac{1}{k}(E^{(k)}_{ij}+\frac{1}{2}L_k\delta_{ij})+(-c)^k\frac{(n-1)!}{(n-2k)!}\delta_{ij}\nonumber\\	 &+\sum_{t=1}^{k-1}(-c)^{k-t}2^{-t}\frac{(n-2t-1)!}{(n-2k)!}\binom{k-1}{t-1}\Bigg(\frac{n-k-t}{t}E^{(t)}_{ij}+\frac{n-2t}{2t}L_t\delta_{ij}\Bigg)\Bigg\}.
	\end{align}
\end{lem}
\begin{proof}
	First, from the definition of $T_{r-1}$, we have
	\begin{align}
	\sum_{m,\alpha}T_{r-1\,mj}^{\alpha}h_{mi}^\alpha&=\frac{1}{(r-1)!}\sum\delta^{i_1\ldots i_{2k-1}m}_{j_1\ldots j_{2k-1}j}\Big(\prod_{s=1}^{k-1}\langle A_{i_{2s-1}j_{2s-1}},A_{i_{2s}j_{2s}}\rangle\Big)\langle A_{i_{r-1}j_{r-1}},A_{mi}\rangle\nonumber\\
	&=:\frac{1}{(r-1)!} Q_{ij}.\label{eq2.37}
	\end{align}
	Now we rewrite the Gauss equation \eqref{gauss} as follows:
	\begin{equation}\label{gauss4}
	\l A_{im}, A_{jl}\r-\l A_{il}, A_{jm}\r=R_{ijml}-(\delta_{im}\delta_{jl}-\delta_{il}\delta_{jm})c,
	\end{equation}
	by setting $(i,j,m,l)=(i_1,i_2,j_1,j_2),\cdots, (i_{2k-1},m,j_{2k-1},i)$ in \eqref{gauss4} respectively, we obtain $k$ equations. We have $\mathrm{LHS}=\mathrm{RHS}$, which are defined by
	\begin{align*}
	\mathrm{LHS}=&\sum\Big(\prod_{s=1}^{k-1}
	\big(\l A_{i_{2s-1}j_{2s-1}}, A_{i_{2s}j_{2s}}\r-\l A_{i_{2s-1}j_{2s}}, A_{i_{2s}j_{2s-1}}\r\big)\Big)\\
	&\cdot\big(\l A_{i_{r-2}j_{r-1}}, A_{mi}\r-\l A_{i_{r-1}i}, A_{j_{r-1}m}\r\big)\delta^{i_1\ldots i_{r-1}m}_{j_1\ldots j_{r-1}j}
	=2^k Q_{ij},\\
	\mathrm{RHS}=&\sum\Big(\prod_{s=1}^{k-1}
	\big(R_{i_{2s-1}i_{2s}j_{2s-1}j_{2s}}-c(\delta_{i_{2s-1}j_{2s-1}}\delta_{i_{2s}j_{2s}}-\delta_{i_{2s-1}j_{2s}} \delta_{i_{2s}j_{2s-1}})\big)\Big)\\
	&\cdot\big(R_{i_{r-1}mj_{r-1}i}-c(\delta_{i_{r-1}j_{r-1}}\delta_{mi}-\delta_{i_{r-1}i} \delta_{mj_{r-1}})\big)\delta^{i_1\ldots i_{r-1}m}_{j_1\ldots j_{r-1}j}.
	\end{align*}
	
	Using \eqref{eq-kron-prop1} and \eqref{eq-kron-prop2}, we have
	\begin{align*}
	\mathrm{RHS_1}:=&\sum\Big(\prod_{s=1}^{k-1}
	\big(R_{i_{2s-1}i_{2s}j_{2s-1}j_{2s}}-c(\delta_{i_{2s-1}j_{2s-1}}\delta_{i_{2s}j_{2s}}-\delta_{i_{2s-1}j_{2s}} \delta_{i_{2s}j_{2s-1}})\big)\Big)\delta^{i_1\ldots i_{r-1}m}_{j_1\ldots j_{r-1}j}\\
	=&\sum_{t=0}^{k-1}(-c)^{k-1-t}\binom{k-1}{t}\Bigg\{\sum\Big(\prod_{s=1}^{t}
	\big(R_{i_{2s-1}i_{2s}j_{2s-1}j_{2s}}\Big)\\
	&\Big(\prod_{s=t+1}^{k-1}(\delta_{i_{2s-1}j_{2s-1}}\delta_{i_{2s}j_{2s}}-\delta_{i_{2s-1}j_{2s}} \delta_{i_{2s}j_{2s-1}})\big)\Big)\delta^{i_1\ldots i_{r-1}m}_{j_1\ldots j_{r-1}j}\Bigg\}\\
	=&\sum_{t=0}^{k-1}(-c)^{k-1-t}\binom{k-1}{t}2^{k-1-t}\frac{(n-2(t+1))!}{(n-r)!}\Bigg\{\sum\Big(\prod_{s=1}^{t}
	\big(R_{i_{2s-1}i_{2s}j_{2s-1}j_{2s}}\Big)\delta^{i_1\ldots i_{2t}i_{r-1}m}_{j_1\ldots j_{2t}j_{r-1}j}\Bigg\}\\
	=&\sum_{t=0}^{k-1}(-c)^{k-1-t}\binom{k-1}{t}2^{k-1-t}\frac{(n-2(t+1))!}{(n-r)!}P_{(t+1)}^{i_{r-1}mj_{r-1}j}.
	\end{align*}
	Combining with Lemma \ref{lem2.2}, we have
	\begin{align}
	\mathrm{RHS}=&\sum \mathrm{RHS_1}\big(R_{i_{r-1}mj_{r-1}i}-c(\delta_{i_{r-1}j_{r-1}}\delta_{mi}-\delta_{i_{r-1}i} \delta_{mj_{r-1}})\big)\nonumber\\
	=&\sum_{t=1}^{k}(-c)^{k-t}\binom{k-1}{t-1}2^{k-t}\frac{(n-2t)!}{(n-r)!}\sum_{s,l,m}P_{(t)}^{slmj}R_{slmi}\nonumber\\
	&-\sum_{t=0}^{k-1}(-c)^{k-t}\binom{k-1}{t}2^{k-t}\frac{(n-2(t+1))!}{(n-r)!}(n-2(t+1)+1)E^{(t)}_{ij}\nonumber\\
	=&\sum_{t=1}^{k}(-c)^{k-t}\binom{k-1}{t-1}2^{k-t}\frac{(n-2t)!}{(n-r)!}\sum_{s,l,m}P_{(t)}^{slmj}R_{slmi}-\sum_{t=1}^{k}(-c)^{k-t+1}\binom{k-1}{t-1}2^{k-t+1}\frac{(n-2t+1)!}{(n-r)!}E^{(t-1)}_{ij}\nonumber\\
	=&\sum_{s,l,m}P_{(k)}^{slmj}R_{slmi}+(-c)^k2^{k}\frac{(n-1)!}{(n-2k)!}\delta_{ij}+\sum_{t=1}^{k-1}(-c)^{k-t}\binom{k-1}{t-1}2^{k-t}\frac{(n-2t)!}{(n-2k)!}\sum_{s,l,m}P_{(t)}^{slmj}R_{slmi}\nonumber\\
	&-\sum_{t=1}^{k-1}(-c)^{k-t}\binom{k-1}{t}2^{k-t}\frac{(n-2t-1)!}{(n-2k)!}E^{(t)}_{ij}\nonumber\\
	=&\sum_{s,l,m}P_{(k)}^{slmj}R_{slmi}+(-c)^k2^{k}\frac{(n-1)!}{(n-2k)!}\delta_{ij}\nonumber\\		 &+\sum_{t=1}^{k-1}(-c)^{k-t}2^{k-t}\frac{(n-2t-1)!}{(n-2k)!}\binom{k-1}{t-1}\Bigg((n-2t)\sum_{s,l,m}P_{(t)}^{slmj}R_{slmi}-\frac{k-t}{t}E^{(t)}_{ij}\Bigg)\nonumber\\
	=&\frac{1}{k}(E^{(k)}_{ij}+\frac{1}{2}L_k\delta_{ij})+(-c)^k2^{k}\frac{(n-1)!}{(n-2k)!}\delta_{ij}\nonumber\\
	&+\sum_{t=1}^{k-1}(-c)^{k-t}2^{k-t}\frac{(n-2t-1)!}{(n-2k)!}\binom{k-1}{t-1}\Bigg(\frac{n-2t}{t}(E^{(t)}_{ij}+\frac{1}{2}L_t\delta_{ij})-\frac{k-t}{t}E^{(t)}_{ij}\Bigg)\nonumber\\
	=&\frac{1}{k}(E^{(k)}_{ij}+\frac{1}{2}L_k\delta_{ij})+(-c)^k2^{k}\frac{(n-1)!}{(n-2k)!}\delta_{ij}\nonumber\\
	&+\sum_{t=1}^{k-1}(-c)^{k-t}2^{k-t}\frac{(n-2t-1)!}{(n-2k)!}\binom{k-1}{t-1}\Bigg(\frac{n-k-t}{t}E^{(t)}_{ij}+\frac{n-2t}{2t}L_t\delta_{ij}\Bigg)
	.\label{eq2-58}
	\end{align}
	Therefore, we get the conclusion from \eqref{eq2.37}, \eqref{eq2-58} and $Q_{ij}=\frac{1}{2^k}\mathrm{LHS}=\frac{1}{2^k}\mathrm{RHS}$.
\end{proof}
\begin{rem}
	R. Reilly obtained the relation in Lemma \ref{lem2.3} in the case $c=0, p=1$, using the Ricci tensor of degree $2k$ (see \cite{Rei73a} for details).
\end{rem}

%\bibliographystyle{abbrv}
%\bibliography{refers}

\end{document}